%% file: cvi-constrained-yamabe.tex
\documentclass{amsart}
\input{header}

\begin{document}

\title[Yamabe problems for polydifferential operators]{Yamabe problems for formally self-adjoint, conformally covariant,  polydifferential operators}
\author{Jeffrey S.\ Case}
\address{Department of Mathematics \\ Penn State University \\ University Park, PA 16802 \\ USA}
\email{jscase@psu.edu}
%\date{\today}
\keywords{conformally invariant operator, conformally variational invariant, Yamabe problem}
\subjclass[2020]{Primary 58J70; Secondary 53C18, 53C21, 58E11}
\begin{abstract}
 Formally self-adjoint, conformally covariant, polydifferential operators provide a general framework for studying variational problems, such as prescribing the scalar, $Q$-, or $\sigma_2$-curvatures, within a conformal class.
 We describe recent progress on Yamabe problems for such operators, including uniqueness results on the sphere and nonuniqueness results in general.
 We also highlight a number of open questions related to these operators, some of which constitute a possible blueprint for the general solution of the Yamabe problem for polydifferential operators.
\end{abstract}
\maketitle

\input{introduction}
\input{bg}
\input{ambient}
\input{frank-lieb}
\input{examples}
\input{nonuniqueness}
\input{future}

\section*{Acknowledgements}
This paper is based on joint works with Jo{\~a}o Henrique Andrade, Opal Cieslak, Yueh-Ju Lin, Paolo Piccione, Yi Wang, Juncheng Wei, Zetian Yan, and Wei Yuan.
I thank them for all of their contributions and our many insightful conversations.

This work was partially supported by a Simons Foundation Collaboration Grant for Mathematicians, ID 524601, and by the National Science Foundation under Award No. DMS-
2505606.

\bibliography{../Bibliography/bib}
\end{document}

%% file: header.tex
\usepackage{amsmath}
\usepackage{amssymb}
\usepackage{amsthm}
\usepackage[msc-links,abbrev]{amsrefs}
\usepackage{hyperref}
\usepackage{mathrsfs}
\usepackage{mathtools}
\usepackage{slashed}
\usepackage{tikz-cd}
\usepackage{enumitem}

\setenumerate{label=(\roman*)}

\DeclareMathOperator{\Id}{Id}

\DeclareMathOperator{\tr}{tr}
\DeclareMathOperator{\Vol}{Vol}
\DeclareMathOperator{\dvol}{dV}

\DeclareMathOperator{\Ric}{Ric}

\DeclareMathOperator{\Rm}{Rm}

\DeclareMathOperator{\Conf}{Conf}

\DeclareMathOperator{\Int}{Int}

\newcommand{\defn}[1]{{\boldmath\bfseries#1}}

\newcommand{\Diff}{\mathrm{Diff}}

\newcommand{\cd}{\widetilde{d}}
\newcommand{\cg}{\widetilde{g}}

\newcommand{\cf}{\widetilde{f}}

\newcommand{\ct}{\widetilde{t}}
\newcommand{\cu}{\widetilde{u}}
\newcommand{\cv}{\widetilde{v}}
\newcommand{\cw}{\widetilde{w}}
\newcommand{\cx}{\widetilde{x}}
\newcommand{\cy}{\widetilde{y}}

\newcommand{\cC}{\widetilde{C}}
\newcommand{\cD}{\widetilde{D}}

\newcommand{\cI}{\widetilde{I}}

\newcommand{\cL}{\widetilde{L}}
\newcommand{\cM}{\widetilde{M}}

\newcommand{\cpi}{\widetilde{\pi}}

\newcommand{\cnabla}{\widetilde{\nabla}}
\newcommand{\cdelta}{\widetilde{\delta}}
\newcommand{\cDelta}{\widetilde{\Delta}}

\newcommand{\cmE}{\widetilde{\mathcal{E}}}

\newcommand{\cmG}{\widetilde{\mathcal{G}}}

\newcommand{\hg}{\widehat{g}}

\newcommand{\hM}{\widehat{M}}

\newcommand{\lv}{\lvert}
\newcommand{\rv}{\rvert}
\newcommand{\lV}{\lVert}
\newcommand{\rV}{\rVert}

\newcommand{\mE}{\mathcal{E}}
\newcommand{\mF}{\mathcal{F}}
\newcommand{\mG}{\mathcal{G}}

\newcommand{\mI}{\mathcal{I}}

\newcommand{\kD}{\mathfrak{D}}

\newcommand{\kc}{\mathfrak{c}}

\newcommand{\bN}{\mathbb{N}}

\newcommand{\bR}{\mathbb{R}}

\newcommand{\sA}{\mathscr{A}}
\newcommand{\sB}{\mathscr{B}}
\newcommand{\sC}{\mathscr{C}}

\def\sideremark#1{\ifvmode\leavevmode\fi\vadjust{\vbox to0pt{\vss
 \hbox to 0pt{\hskip\hsize\hskip1em
 \vbox{\hsize3cm\tiny\raggedright\pretolerance10000
 \noindent #1\hfill}\hss}\vbox to8pt{\vfil}\vss}}}

\newcommand{\suchthat}{\mathrel{}:\mathrel{}}

\newtheorem{theorem}{Theorem}[section]
\newtheorem{proposition}[theorem]{Proposition}
\newtheorem{lemma}[theorem]{Lemma}

\theoremstyle{definition}
\newtheorem{definition}[theorem]{Definition}

\newtheorem{objective}[theorem]{Objective}
\newtheorem{conjecture}[theorem]{Conjecture}
\newtheorem{example}[theorem]{Example}

\theoremstyle{remark}
\newtheorem{remark}[theorem]{Remark}

\numberwithin{equation}{section}

%% file: introduction.tex
\section{Introduction}
\label{sec:intro}

Variational problems play a fundamental role in geometry and physics.
The key motivating example for this paper is the Yamabe Problem~\cite{LeeParker1987}, which asks to conformally change a metric to have constant scalar curvature.
Equivalently\footnote{
 The metric $u^{4/(n-2)}g$ has constant scalar curvature $\lambda$ if and only if $u$ is a smooth positive solution of PDE~\eqref{eqn:yamabe-pde}.
}, one seeks positive smooth solutions of the semilinear PDE\footnote{
 Our convention is that $\Delta \geq 0$ on functions.
}
\begin{equation}
 \label{eqn:yamabe-pde}
 \Delta u + \frac{n-2}{4(n-1)}Ru = \frac{n-2}{4(n-1)} \lambda \lv u \rv^{\frac{4}{n-2}}u .
\end{equation}
Solutions to the Yamabe Problem always exist on compact manifolds;
they can be constructed by minimizing the volume-normalized total scalar curvature functional~\cites{Yamabe1960,Aubin1976,Schoen1984,Trudinger1968} or by using min-max methods to find (possibly) high energy solutions~\cites{Bahri1993,BahriBrezis1996}.
Schoen~\cite{Schoen1989} exhibited pairwise nonhomothetic\footnote{
 Two metrics $g_1$ and $g_2$ are \defn{homothetic} if there is a diffeomorphism $\Phi$ and a constant $c>0$ such that $g_1 = c^2\Phi^\ast g_2$, and \defn{nonhomothetic} otherwise.
 For example, the $(n+2)$-parameter family of solutions to PDE~\eqref{eqn:yamabe-pde} on the round $n$-sphere are all homothetic~\cite{Obata1971}.}
solutions to the Yamabe Problem on $S^1 \times S^{n-1}$ by studying $S^{n-1}$-symmetric solutions of PDE~\eqref{eqn:yamabe-pde};
this nonuniqueness result was substantially generalized by Bettiol and Piccione~\cite{BettiolPiccione2018} via a covering argument.
These arguments produce solutions of PDE~\eqref{eqn:yamabe-pde} that are not energy minimizers.
Solutions to the Yamabe Problem need not exist on noncompact manifolds~\cites{Pohozaev1965,Jin1988}, though they do exist on manifolds whose behavior at infinity is suitably controlled (e.g.\ \cites{AvilesMcOwen1988,MazzeoPacard1999}).
Lifting the metrics produced by Schoen and Bettiol--Piccione to the universal cover shows that solutions to the Yamabe Problem on noncompact manifolds are also generally not unique.

The purpose of this paper is to describe a general framework for studying Yamabe problems for variational scalar invariants in conformal geometry.
These problems include both semilinear problems---such as for the $Q$-curvature (cf.\ \cite{HangYang2016l})---and fully nonlinear problems---such as for the $\sigma_k$-curvatures on locally conformally flat manifolds (cf.\ \cite{ShengTrudingerWang2012}).
To simultaneously emphasize the geometric and analytic aspects of this framework, we focus on Yamabe problems for formally self-adjoint, conformally covariant, polydifferential operators.
This paper is primarily a survey, though its streamlined discussion of examples, its focus on constrained polydifferential operators  (see Definition~\ref{defn:constrained-operator}), and the lemmas which develop their basic properties are new.
We also raise some conjectures and objectives that we expect to drive future progress towards the solution of these Yamabe problems.

A \defn{formally self-adjoint, conformally covariant, polydifferential operator} of \defn{rank} $r \in \bN$ on $n$-manifolds is a function $D$ that assigns to each Riemannian $n$-manifold $(M,g)$ a linear operator
\begin{equation*}
 D^g \colon C^\infty(M)^{\otimes(r-1)} \to C^\infty(M)
\end{equation*}
such that
\begin{enumerate}
 \item $D^g(u_1 \otimes \dotsm \otimes u_{r-1})$ is expressed by a universal\footnote{
  A linear combination is \defn{universal} if the coefficients are independent of $(M,g)$.}
 linear combination of complete contractions of tensor products of covariant derivatives of the Riemann curvature tensor, covariant derivatives of the functions $u_1,\dotsc,u_{r-1}$, and the inverse metric;
 \item the \defn{Dirichlet form} $\kD^g \colon C_0^\infty(M)^{\otimes r} \to \bR$ determined\footnote{
  Here $C_0^\infty(M)$ is the set of compactly-supported smooth functions on $M$.
 }
 by
 \begin{equation*}
  \kD^g( u_1 \otimes \dotsm \otimes u_{r} ) := \int_M u_1 D^g( u_2 \otimes \dotsm \otimes u_{r} ) \dvol_g
 \end{equation*}
 is symmetric; and
 \item there is a weight $w \in \bR$ such that if $\hg = e^{2\Upsilon}g$, then
 \begin{equation*}
  \kD^{\hg} (e^{w\Upsilon}u_1 \otimes \dotsm \otimes e^{w\Upsilon}u_{r} ) = \kD^g ( u_1 \otimes \dotsm \otimes u_{r} ) .
 \end{equation*}
\end{enumerate}
All tensor products in this paper are over $\bR$.
These assumptions imply that there is a constant $-2k \in \{ 0 , -2 , -4, \dotsc \}$, called the \defn{homogeneity} of $D$, such that $D^{c^2g} = c^{-2k}D^g$ for all Riemannian $n$-manifolds $(M,g)$ and all constants $c>0$.
Moreover, $rw = 2k - n$, and hence the weight is determined by the dimension, rank, and homogeneity of $D$.
The most familiar examples are the GJMS operators discovered by Graham, Jenne, Mason, and Sparling~\cite{GJMS1992};
they have rank $2$.
See Section~\ref{sec:bg} for further discussion, and see Table~\ref{table:known-results} and Section~\ref{sec:ambient} for examples.

Let $D$ be a formally self-adjoint, conformally covariant, polydifferential operator of rank $r$ and homogeneity $-2k \not= -n$ on $n$-manifolds.
Define $I$ by
\begin{equation}
 \label{eqn:associated-operator}
 D^g( 1^{\otimes(r-1)} ) = \left( \frac{n-2k}{r} \right)^{r-1} I^g .
\end{equation}
Then $I^{c^2g} = c^{-2k}I^g$ for all constants $c>0$ and
\begin{equation*}
 \left. \frac{d}{dt} \right|_{t=0} \frac{1}{n-2k}\int_M I^{e^{2tu}g} \dvol_{e^{2tu}g} = \int_M uI^g \dvol_g .
\end{equation*}
In particular, $I$ is a CVI\footnote{
 A \defn{CVI}, or conformally variational invariant, is a homogeneous scalar Riemannian invariant $I$ that is the $L^2$-gradient of a Riemannian functional~\cite{CaseLinYuan2016}.
}.
Every CVI arises this way~\cite{CaseLinYuan2018b}:

\begin{theorem}
 \label{thm:cly}
 Let $I$ be a CVI of homogeneity $-2k \not= -n$ on $n$-manifolds.
 There is a minimal integer $r \leq 2k$ for which there exists a formally self-adjoint, conformally covariant, polydifferential operator $D$ of rank $r$ satisfying Condition~\eqref{eqn:associated-operator}.
 Moreover, $D$ is unique among all such operators of rank $r$.
\end{theorem}

The first statement of Theorem~\ref{thm:cly} is due to Case et al.~\cite{CaseLinYuan2018b}*{Theorem~1.6}, and the second is new.
There is a corresponding statement when $n=2k$~\cite{CaseLinYuan2018b}.

A formally self-adjoint, conformally covariant, polydifferential operator $D$ has \defn{minimal rank} if it is an operator produced by Theorem~\ref{thm:cly}.
There are such operators that do not have minimal rank;
e.g.\ Example~\ref{ex:nonminimal-gjms-operator}.

The proof of Theorem~\ref{thm:cly} is algorithmic.
Since there is also an algorithm producing a spanning set for the space of CVIs~\cite{CaseLinYuan2016}*{Section~8}, one can, in principle, produce a spanning set for the space of formally self-adjoint, conformally covariant, polydifferential operators of minimal rank.
However, this construction is impractical for studying many analytic properties discussed in this paper.
Objective~\ref{o:classification} seeks an alternative description that better captures these properties.

Let $D$ and $I$ be as in Condition~\eqref{eqn:associated-operator}.
Then $u \in C^\infty(M)$ solves
\begin{equation}
 \label{eqn:pde}
 D^g(u^{\otimes(r-1)}) = \left(\frac{n-2k}{r}\right)^{r-1}\lambda \lv u \rv^{\frac{(r-2)n+4k}{n-2k}}u
\end{equation}
if and only if $\hg := \lv u \rv^{2r/(n-2k)}g$ satisfies
\begin{equation*}
 I^{\hg} = \left( \lv u \rv / u \right)^r\lambda \quad\text{on}\quad \left\{ x \in M \suchthat u(x) \not= 0 \right\} .
\end{equation*}
Formal self-adjointness and conformal covariance make PDE~\eqref{eqn:pde} variational:

\begin{lemma}
 \label{lem:euler-equation}
 Let $D$ be a formally self-adjoint, conformally covariant, polydifferential operator of rank $r$ and homogeneity $-2k \not= -n$ on $n$-manifolds.
 Given a compact Riemannian $n$-manifold $(M,g)$, define $\mF^g \colon C^\infty(M) \setminus \{ 0 \} \to \bR$ by
 \begin{equation}
  \label{eqn:defn-mF}
  \mF^g(u) := \left. \kD^g\bigl(u^{\otimes r}\bigr) \middle/ \lV u \rV_{r^\ast}^{r} \right. ,
 \end{equation}
 where $r^\ast := rn/(n-2k)$ and $\lV \cdot \rV_{r^\ast}$ denotes the $L^{r^\ast}$-norm.
 Then $u$ is a critical point of $\mF^g$ if and only if $u$ solves PDE~\eqref{eqn:pde} for some $\lambda \in \bR$.
\end{lemma}

Since PDE~\eqref{eqn:pde} is nonlinear in general, it is typically not elliptic for all $u$.
This motivates one to associate constraints to $D$:

\begin{definition}
 \label{defn:constrained-operator}
 A \defn{constrained polydifferential operator} of rank $r$ and homogeneity $-2k$ is a pair $(D,\sC)$ of a formally self-adjoint, conformally covariant, polydifferential operator $D$ of rank $r$ and homogeneity $-2k$ on $n$-manifolds and a (possibly empty) finite set $\sC$ of symmetric conformally covariant polydifferential operators such that for each $C \in \sC$ it holds that
 \begin{equation*}
  C^{e^{2\Upsilon}g}\bigl(u^{\otimes (r(C)-1)}\bigr) = e^{b(C)\Upsilon} C^g\left( \bigl( e^{\frac{n-2k}{r}\Upsilon}u \bigr)^{\otimes(r(C)-1)} \right)
 \end{equation*}
 for all Riemannian $n$-manifolds $(M,g)$ and all $\Upsilon,u \in C^\infty(M)$, where $r(C)$ is the rank of $C$ and $b(C) \in \bR$ is determined by $C$.
\end{definition}

See Section~\ref{sec:bg} for an explanation of the terminology used in this definition.
See Table~\ref{table:known-results} and Section~\ref{sec:examples} for examples of constrained polydifferential operators.

The key point of Definition~\ref{defn:constrained-operator} is the conformal transformation property of elements of the constraint set $\sC$.
This has the effect that constrained polydifferential operators determine conformally invariant cones in $C^\infty(M)$:

\begin{lemma}
 \label{lem:geometric-cone}
 Let $(D,\sC)$ be a constrained polydifferential operator of rank $r$ and homogeneity $-2k$ on $n$-manifolds.
 Given a Riemannian $n$-manifold $(M,g)$, set
 \begin{equation}
  \label{eqn:geometric-cone}
  U_{\sC}^g := \left\{ u \in C^\infty(M) \suchthat \forall C \in \sC , C^g\bigl(u^{\otimes(r(C)-1)}\bigr) > 0 \right\} .
 \end{equation}
 The sets $U_\sC^g$ have the following properties:
 \begin{enumerate}
  \item If $u \in U_{\sC}^g$ and $c>0$, then $cu \in U_{\sC}^g$.
  \item If $\hg = e^{2\Upsilon}g$, then $U_{\sC}^{\hg} = e^{-\frac{n-2k}{r}\Upsilon}U_{\sC}^g$.
  \item If $f \colon \hM \to M$ is a local diffeomorphism, then $f^\ast(U_{\sC}^g) \subseteq U_{\sC}^{f^\ast g}$.
 \end{enumerate}
\end{lemma}

In particular, the \defn{constrained cone} $U_\sC$, which is an assignment to each Riemannian $n$-manifold of a subset of its space of smooth functions, is a geometric cone of weight $-\frac{n-2k}{r}$ in the sense of Andrade et al.~\cite{AndradeCasePiccioneWei2023}.
The last two properties in Lemma~\ref{lem:geometric-cone} give rise to an action of the conformal group\footnote{
 The \defn{conformal group} $\Conf(M,\kc)$ is the group, under composition, of all diffeomorphisms $\Phi \in \Diff(M)$ such that $\Phi^\ast g \in \kc$ for some, and hence any, $g \in \kc$.
}
on $U_\sC$ that fixes the set of solutions of PDE~\eqref{eqn:pde}:

\begin{lemma}
 \label{lem:action}
 Let $(D,\sC)$ be a constrained polydifferential operator of rank $r$ and homogeneity $-2k > -n$ on $n$-manifolds.
 Fix a Riemannian $n$-manifold $(M,g)$ and set $w := -(n-2k)/r$.
 Then
 \begin{equation}
  \label{eqn:action}
  u \cdot \Phi := \lv J_\Phi \rv^{-w/n} \Phi^\ast u
 \end{equation}
 defines\footnote{
  The Jacobian determinant $\lv J_\Phi \rv$ is defined by $\dvol_{\Phi^\ast g} = \lv J_\Phi \rv \dvol_g$.
 } a right action of $\Conf(M,\kc)$ of $C^\infty(M)$ and $U_{\sC}^g$ is invariant under this action.
 Moreover, if $u \in U_{\sC}^g$ solves PDE~\eqref{eqn:pde}, then $u \cdot \Phi$ also solves PDE~\eqref{eqn:pde}.
\end{lemma}

Note that the action~\eqref{eqn:action} depends on $(D,\sC)$ through the role of the weight $w$.

In this paper we focus on the existence of solutions of PDE~\eqref{eqn:pde} as \emph{minimizers} of a constrained variational problem:

\begin{definition}
 \label{defn:yamabe}
 Let $(D,\sC)$ be a constrained polydifferential operator of rank $r$ and homogeneity $-2k>-n$ on $n$-manifolds.
 The \defn{$(D,\sC)$-Yamabe constant} of a compact conformal $n$-manifold $(M,\kc)$ is
 \begin{equation}
  \label{eqn:yamabe-defn-by-normalization}
  Y_{(D,\sC)}(M,\kc) := \inf \left\{ \kD^g\bigl(u^{\otimes r}\bigr) \suchthat u \in U_{\sC}^g , \lV u \rV_{r^\ast} = 1 \right\} ,
 \end{equation}
 where $g \in \kc$ and $r^\ast := rn/(n-2k)$.
 
 The \defn{$(D,\sC)$-Yamabe Problem} is to find a positive $u \in U_{\sC}^g$ such that
 \begin{equation*}
  \kD^g\bigl(u^{\otimes r}\bigr) = Y_{(D,\sC)}(M,g) \quad\text{and}\quad \lV u \rV_{r^\ast} = 1 .
 \end{equation*}
\end{definition}

For an unconstrained polydifferential operator $D$, we call these the $D$-Yamabe constant and the $D$-Yamabe problem, respectively.

The properties of $D$ and $\sC$ ensure that $Y_{(D,\sC)}(M,\kc)$ is independent of the choice of metric $g \in \kc$.
By homogeneity,
\begin{equation}
 \label{eqn:yamabe-defn-by-homogeneity}
 Y_{(D,\sC)}(M,\kc) = \inf \left\{ \mF^g(u) \suchthat 0 \not= u \in U_{\sC}^g \right\} ,
\end{equation}
where $\mF$ is as in Lemma~\ref{lem:euler-equation} and $g \in \kc$.
If $u$ is a positive minimizer of $Y_{(D,\sC)}(M,\kc)$, then $u^{2r/(n-2k)}g$ has constant $I$-curvature for $I$ as in Equation~\eqref{eqn:associated-operator}.

We have not yet commented on analytic criteria that the constraints $\sC$ should satisfy.
One criterion should be that PDE~\eqref{eqn:pde} is elliptic for every $u \in U_{\sC}^g$.
Another criterion emerges from the $(D,\sC)$-Yamabe Problem on the \defn{standard conformal $n$-sphere} $(S^n,\kc_0)$;
i.e.\ on $S^n$ with its flat conformal structure:
Frank and Lieb~\cite{FrankLieb2012b} gave a rearrangement-free\footnote{
 Their strategy also works in the CR setting, where rearrangement methods are not available.
 See Footnote~\ref{cr-footnote} for further discussion of possible CR analogues of this paper.
}
proof of the classification of minimizers for the $L_{2k}$-Yamabe Problem---i.e.\ extremals for the sharp Sobolev inequality for the embedding $W^{k,2}(\bR^n) \hookrightarrow L^{2^\ast}(\bR^n)$---by using the second variation of the functional $\mF$ and a sharp spectral inequality for $L_{2k}$.
Case~\cite{Case2019fl} isolated the key property needed to extend their argument to polydifferential operators:

\begin{definition}
 \label{defn:frank-lieb}
 A constrained polydifferential operator $(D,\sC)$ of rank $r$ and homogeneity $-2k>-n$ on $n$-manifolds satisfies the \defn{Frank--Lieb Property} if
 \begin{enumerate}
  \item $1 \in U_{\sC}^{d\theta^2}$, where $(S^n,d\theta^2)$ is the round $n$-sphere; and
  \item if $u \in U_{\sC}^{d\theta^2}$, then
  \begin{equation}
   \label{eqn:fl-poincare}
   \kD^{d\theta^2}\bigl(u^{\otimes r}\bigr) \geq \frac{(r-1)(n-2k)}{2rk}\int_{S^n} \sum_{i=0}^{n} x^iu[D,x^i]_1\bigl( u^{\otimes(r-1)} \bigr) \dvol_{d\theta^2}
  \end{equation}
  with equality if and only if $u$ is constant.
 \end{enumerate}
\end{definition}

Here $[D,x]_1 \colon C^\infty(M)^{\otimes(r-1)} \to C^\infty(M)$ is the commutator
\begin{equation*}
 [D,x]_1(u_1 \otimes \dotsm \otimes u_{r-1}) := D( xu_1 \otimes u_2 \otimes \dotsm \otimes u_{r-1} ) - xD( u_1 \otimes \dotsm \otimes u_{r-1} )
\end{equation*}
and $x^0,\dotsc,x^n$ are the restrictions of the standard Cartesian coordinates to $S^n$.

The Frank--Lieb Property is a sufficient condition for the uniqueness of minimizers of the $(D,\sC)$-Yamabe Problem on the round $n$-sphere~\cite{Case2019fl}:

\begin{theorem}
 \label{thm:frank-lieb}
 Let $(D,\sC)$ be a constrained polydifferential operator of rank $r$ and homogeneity $-2k>-n$ on $n$-manifolds that satisfies the Frank--Lieb Property.
 If $u$ is a local minimizer of the functional $\mF^{d\theta^2} \colon U_{\sC}^{d\theta^2} \setminus \{ 0 \} \to \bR$ defined by Equation~\eqref{eqn:defn-mF}, then there is a constant $c \not= 0$ and a $\Phi \in \Conf(S^n)$ such that
 \begin{equation*}
  u = c \cdot \Phi .
 \end{equation*}
\end{theorem}

Theorem~\ref{thm:frank-lieb} is due to Case~\cite{Case2019fl}*{Proposition~4.1}, though our treatment adds the role of constrained polydifferential operators and allows $u$ to change sign.

Case~\cite{Case2019fl} gave a simple proof that the GJMS operators and the constrained $\sigma_2$-operator satisfy the Frank--Lieb Property.
A key point is that these are \defn{Weyl operators}---that is, they can be defined via a tangential operator in the Fefferman--Graham ambient space~\cite{FeffermanGraham2012}---and the commutator is easily computed holographically;
see Section~\ref{sec:frank-lieb}.
Table~\ref{table:known-results} specifies which constrained operators are known to satisfy the Frank--Lieb Property.
Objectives~\ref{o:classification} and~\ref{o:commutator} together aim to use the Frank--Lieb Property to compute $Y_{(D,\sC)}(S^n,\kc_0)$.

Since a geometric cone $U_{\sC}^g$ can be empty on a given Riemannian $n$-manifold\footnote{
 The cone $U_{\sC}^g$ in Example~\ref{ex:sigma2-commutator} is empty on manifolds with nonpositive Yamabe constant.
},
a global assumption is necessary to solve the $(D,\sC)$-Yamabe Problem.
Andrade et al.~\cite{AndradeCasePiccioneWei2023} proposed that any such assumption should satisfy the following properties:

\begin{definition}
 \label{defn:aubin}
 Let $(D,\sC)$ be a constrained polydifferential operator of rank $r$ and homogeneity $-2k>-n$ on $n$-manifolds.
 A \defn{geometric Aubin set} for $(D,\sC)$ is a set $\sA$ of compact conformal $n$-manifolds such that
 \begin{enumerate}
  \item if $(M,\kc) \in \sA$, then $0 < Y_{(D,\sC)}(M,\kc) \leq Y_{(D,\sC)}(S^n,\kc_0)$ with equality if and only if $(M,\kc)$ is conformally equivalent to $(S^n,\kc_0)$;
  \item the $(D,\sC)$-Yamabe Problem has a solution for all $(M,\kc) \in \sA$; and
  \item if $(M,\kc) \in \sA$ and $\pi \colon \cM \to M$ is a finite covering, then $(\cM,\pi^\ast\kc) \in \sA$.
 \end{enumerate}
\end{definition}

Here $\pi^\ast\kc := [ \pi^\ast g]$ is the conformal class on $\cM$ induced by $\pi$;
it is independent of the choice of $g \in \kc$.
The positivity of the $(D,\sC)$-Yamabe constant is equivalent, under the assumption of a Sobolev-type inequality, to the positivity of the first nonlinear eigenvalue of $(D,\sC)$;
see Lemma~\ref{lem:positivity-to-spectrum}.

Prior analytic work of many authors identifies nonempty geometric Aubin sets for the $Q$-curvatures and the $\sigma_k$-curvatures when they are variational;
see Section~\ref{sec:examples}.
This is summarized in Table~\ref{table:known-results}.

\begin{table}[h!]
 \label{table:known-results}
 \caption{Known constrained Weyl operators, the Frank--Lieb Property, and their geometric Aubin sets}
 \renewcommand{\arraystretch}{1.1}
 \begin{tabular}{|c|c|c|c|c|c|}
  \hline
  Operator & Rank & Order & Constraint & Frank--Lieb? & Aubin Set \\
  \hline\hline
  GJMS & $2$ & $2k$ & $\emptyset$ & Yes & Yes, $k \leq 2$ \\
  Ex.~\ref{ex:gjms-operator} & & & & Ex.~\ref{ex:gjms-commutator} & cf.\ Conj.~\ref{conj:gjms-aubin-set} \\
  \hline
  $\sigma_2$-operator & $4$ & $2$ & $\sigma_1$-operator & Yes & Partial \\
  Ex.~\ref{ex:sigma2-operator} & & & & Ex.~\ref{ex:sigma2-commutator} & Ex.~\ref{ex:sigma2-set} \\
  \hline
  $v_k$-operator & $2k$ & $2$ & Partial & Partial & Partial \\
  Ex.~\ref{ex:vk} & & & & Ex.~\ref{ex:vk-commutator} & Ex.~\ref{ex:vk-set} \\
  \hline
  Ovsienko--Redou & $3$ & $2k$ & $\Id$ & Yes, $k = 1,2$ & ? \\
  Ex.~\ref{ex:ovsienko-redou-operator} & & & & Ex.~\ref{ex:ovsienko-redou-commutator} & \\
  \hline
 \end{tabular}
\end{table}

There are two key points to Definition~\ref{defn:aubin}.
First, it reflects the observation that many (e.g.\ \cites{LeeParker1987,HangYang2016l,ShengTrudingerWang2012,Aubin1976}) Yamabe-type problems can be easily solved under the assumption $Y_{(D,\sC)}(M^n,\kc) < Y_{(D,\sC)}(S^n,\kc_0)$.
Second, it is closed under finite coverings and implies the existence of solutions of PDE~\eqref{eqn:pde} with uniformly bounded energy.
Andrade et al.~\cite{AndradeCasePiccioneWei2023} showed that these properties lead to geometrically distinct\footnote{
 Solutions $u_1,u_2 \in U$ of PDE~\eqref{eqn:pde} are \defn{geometrically distinct} if $u_1 \not\in u_2 \cdot \Conf(M,\kc)$.
}
solutions to PDE~\eqref{eqn:pde}:

\begin{theorem}
 \label{thm:compact-nonuniqueness}
 Let $\sA$ be a geometric Aubin set for a constrained polydifferential operator $(D,\sC)$.
 Suppose $(M,\kc) \in \sA$ is such that $\pi_1(M)$ has infinite profinite completion.
 For each $m \in \bN$, there is a finite regular covering $\pi \colon \cM \to M$ such that for any $g \in \kc$ there are at least $m$ positive, pairwise geometrically distinct solutions of PDE~\eqref{eqn:pde} in $U_{\sC}^{\pi^\ast g}$.
\end{theorem}

Theorem~\ref{thm:compact-nonuniqueness} differs from its original formulation~\cite{AndradeCasePiccioneWei2023}*{Theorem~1.2} only in its use of constrained polydifferential operators.
In Section~\ref{sec:nonuniqueness} we prove a stronger result that requires only the positivity of the $(D,\sC)$-Yamabe constant and the existence of critical points of the $\mF$-functional of uniformly bounded energy;
that such a statement holds was briefly mentioned by Andrade et al.~\cite{AndradeCasePiccioneWei2023}.
This formulation can be applied to critical points of $\mF$ constructed by flow methods (e.g.\ \cites{GurskyMalchiodi2014,ShengTrudingerWang2007}).

The assumption that $\pi_1(M)$ has infinite profinite completion is equivalent to the existence of an infinite tower
\begin{equation*}
 \dotsm \to \cM_3 \to \cM_2 \to \cM_1 \to M
\end{equation*}
of finite connected coverings of degree at least two~\cite{BettiolPiccione2018}.
The Selberg--Malcev Lemma~\cite{Ratcliffe2006} implies that Theorem~\ref{thm:compact-nonuniqueness} applies to products $(H^k / \Gamma) \times S^{n-k}$ of compact hyperbolic manifolds and spheres when they are in a geometric Aubin set.

The solutions of PDE~\eqref{eqn:pde} constructed by Theorem~\ref{thm:compact-nonuniqueness} all lift to the conformal universal cover $(\cM,\cpi^\ast\kc)$ of $(M,\kc)$.
Since $\Conf(\cM,\cpi^\ast\kc)$ can contain elements that do not descend to the quotient,\footnote{
 For example, the conformal diffeomorphism group of hyperbolic space is noncompact, but compact hyperbolic manifolds have finite conformal group~\cite{Ratcliffe2006}.
}
it is not clear whether the lifts of geometrically distinct solutions remain geometrically distinct.
Andrade et al.~\cite{AndradeCasePiccioneWei2023} addressed this issue using the Ferrand--Obata Theorem~\cites{Ferrand1996,Obata1971,Schoen1995}:

\begin{theorem}
 \label{thm:noncompact-nonuniqueness}
 Let $\sA$ be a geometric Aubin set for a constrained polydifferential operator $(D,\sC)$.
 Suppose $(M,\kc) \in \sA$ is such that $\pi_1(M)$ has infinite profinite completion.
 If the conformal universal cover $(\cM,\cpi^\ast\kc)$ of $(M,\kc)$ is not conformally equivalent to $(\bR^n,\kc_0)$, then for any $g \in \kc$ there is a countable set $\{ \cu_j \}_{j\in \bN}$ of positive, pairwise geometrically distinct solutions of PDE~\eqref{eqn:pde} in $U_{\sC}^{\cpi^\ast g}$.
\end{theorem}

Theorem~\ref{thm:noncompact-nonuniqueness} differs from its original formulation~\cite{AndradeCasePiccioneWei2023}*{Theorem~1.3} only in its use of constrained polydifferential operators.
As such, we do not give a proof.

Since the only compact quotients of $(\bR^n,\kc_0)$ are (finite quotients of conformal classes of) flat tori,\footnote{
    Liouville's Theorem and the requirement that $M^n = \bR^n/\Gamma$ is smooth and compact imply that $\Gamma \subset \Conf(\bR^n,\kc_0)$ must be a subset of the isometry group of the flat metric.
}
the assumption on the conformal universal cover can be replaced by the assumption that $(M,\kc)$ has positive Yamabe constant.

Since $H^m \times S^{n-m}$ and $S^n \setminus S^{m-1}$ are conformally equivalent with their flat conformal structure~\cites{MazzeoSmale1991}, Theorem~\ref{thm:noncompact-nonuniqueness} produces infinitely many geometrically distinct \emph{singular} solutions of PDE~\eqref{eqn:pde} on $S^n$ for many constrained polydifferential operators $(D,\sC)$ (cf.\ \cites{BettiolPiccione2018,BettiolPiccioneSire2021,MazzeoPacard1996}).

Theorems~\ref{thm:compact-nonuniqueness} and~\ref{thm:noncompact-nonuniqueness} indicate the utility of the general framework of formally self-adjoint, conformally covariant, polydifferential operators.
In Section~\ref{sec:future} we present a series of objectives that aim to compute $Y_{(D,\sC)}(S^n,\kc_0)$ and to solve the $(D,\sC)$-Yamabe Problem.
We also pose Conjecture~\ref{conj:gjms-einstein}, which is a variant of the longstanding conjectural classification~\cites{AIM2003-Conference,Branson2005} of GJMS operators, and Conjectures~\ref{conj:gjms-aubin-set} and~\ref{conj:case-malchiodi}, which seek geometric Aubin sets for the GJMS operators.

This paper is organized as follows:

In Section~\ref{sec:bg} we discuss natural polydifferential operators in more detail, prove Lemmas~\ref{lem:euler-equation}, \ref{lem:geometric-cone}, and \ref{lem:action}, and sketch the proof of Theorem~\ref{thm:cly}.

In Section~\ref{sec:ambient} we present constructions of the Weyl operators in Table~\ref{table:known-results}.

In Section~\ref{sec:frank-lieb} we prove Theorem~\ref{thm:frank-lieb}, explain how to holographically compute the commutator in Condition~\eqref{eqn:fl-poincare}, and apply this to the operators in Table~\ref{table:known-results}.

In Section~\ref{sec:examples} we relate the positivity of the $(D,\sC)$-Yamabe constant to a nonlinear eigenvalue.
We also describe geometric Aubin sets for the operators in Table~\ref{table:known-results}.

In Section~\ref{sec:nonuniqueness} we prove a stronger version of Theorem~\ref{thm:compact-nonuniqueness} that relaxes the assumption that the $(D,\sC)$-Yamabe Problem is solvable.

In Section~\ref{sec:future} we present objectives that aim at the $(D,\sC)$-Yamabe Problem.

%% file: bg.tex
\section{Conformally covariant polydifferential operators}
\label{sec:bg}

A \defn{(natural) polydifferential operator} of \defn{rank} $\ell + 1 \in \bN$ on $n$-manifolds is an assignment $D$ to each Riemannian $n$-manifold $(M,g)$ of a linear operator
\begin{equation*}
 D^g \colon C^\infty(M)^{\otimes\ell} \to C^\infty(M)
\end{equation*}
such that $D^g(u_1 \otimes \dotsm \otimes u_\ell)$ can be universally expressed as an $\bR$-linear combination of complete contractions of tensors
\begin{equation}
 \label{eqn:building-blocks}
 \nabla^{N_1}\Rm \otimes \dotsm \otimes \nabla^{N_k}\Rm \otimes \nabla^{N_{k+1}}u_1 \otimes \dotsm \otimes \nabla^{N_{k+\ell}}u_\ell ,
\end{equation}
where $N_1,\dotsc,N_{k+\ell} \in \bN \cup \{ 0 \}$ are powers.
\defn{Scalar Riemannian invariants} are polydifferential operators of rank $1$.
We say that $D$ has \defn{homogeneity} $h \in \bR$ if $D^{c^2g} = c^hD^g$ for all Riemannian $n$-manifolds $(M,g)$ and all constants $c>0$.
Any complete contraction of the tensor~\eqref{eqn:building-blocks} has homogeneity $-2k-\sum N_j$, and a complete contraction can only be formed if $\sum N_j$ is even.
Hence the homogeneity, when defined, is a nonpositive even integer.
We say that $D$ is \defn{symmetric} if
\begin{equation*}
 D^g\bigl( u_{\sigma(1)} \otimes \dotsm \otimes u_{\sigma(\ell)}\bigr) = D^g\bigl( u_1 \otimes \dotsm \otimes u_\ell\bigr) 
\end{equation*}
for all Riemannian $n$-manifolds $(M,g)$, all $u_1,\dotsc,u_\ell \in C^\infty(M)$, and all permutations $\sigma \in S_\ell$.

The \defn{Dirichlet form} associated to a polydifferential operator $D$ of rank $r$ is the assignment $\kD$ to each Riemannian $n$-manifold $(M,g)$ of a linear map
\begin{equation*}
 \kD^g \colon C_0^\infty(M)^{\otimes r} \to \bR
\end{equation*}
determined by
\begin{equation*}
 \kD^g( u_1 \otimes \dotsm \otimes u_r) := \int_M u_1 D^g( u_2 \otimes \dotsm \otimes u_r ) \dvol_g ,
\end{equation*}
where $\dvol_g$ is the Riemannian volume element determined by $g$.
We say that $D$ is \defn{formally self-adjoint} if $\kD$ is symmetric;
i.e.
\begin{equation*}
 \kD( u_{\sigma(1)} \otimes \dotsm \otimes u_{\sigma(r)} ) = \kD( u_1 \otimes \dotsm \otimes u_r )
\end{equation*}
for all $u_1,\dotsc,u_r \in C_0^\infty(M)$ and all permutations $\sigma \in S_r$.
Note that if $D$ is formally self-adjoint, then it is symmetric.

Formal self-adjointness implies that PDE~\eqref{eqn:pde} is variational:

\begin{proof}[Proof of Lemma~\ref{lem:euler-equation}]
 Let $u,v \in C^\infty(M)$.
 Since $D$ is formally self-adjoint,
 \begin{equation*}
  \left. \frac{d}{dt} \right|_{t=0} \kD^g\bigl((u+tv)^{\otimes r}\bigr) = r\kD\bigl( v \otimes u^{\otimes(r-1)}\bigr) = \int_M rvD^g\bigl(u^{\otimes(r-1)}\bigr) \dvol_g .
 \end{equation*}
 It follows that if $u \not= 0$, then
 \begin{equation*}
  \left. \frac{d}{dt} \right|_{t=0} \mF^g(u+tv) = \frac{r}{\lV u \rV_{r^\ast}^r}\int_M v\left( D^g\bigl(u^{\otimes(r-1)}\bigr) - \frac{\kD^g\bigl(u^{\otimes r}\bigr)}{\lV u \rV_{r^\ast}^{r^\ast}}\lv u \rv^{r^\ast-2}u \right) \dvol_g .
 \end{equation*}
 The conclusion readily follows.
\end{proof}

A polydifferential operator $D$ of rank $\ell+1$ is \defn{conformally covariant} if there are weights $a_1,\dotsc,a_\ell,b \in \bR$ such that if $g$ and $\hg = e^{2\Upsilon}g$ are conformally related metrics on an $n$-manifold $M$, then
\begin{equation} 
 \label{eqn:conformal-covariant-bidegree}
 D^{\hg}( u_1 \otimes \dotsm \otimes u_\ell ) = e^{b\Upsilon}D^g ( e^{-a_1\Upsilon}u_1 \otimes \dotsm \otimes e^{-a_\ell \Upsilon}u_\ell )
\end{equation}
for all $u_1,\dotsc,u_\ell \in C^\infty(M)$.
Such an operator has homogeneity $h = b - \sum_{j=1}^\ell a_j$.
Moreover, the Dirichlet forms are related by
\begin{equation*}
 \kD^{\hg}(u_0 \otimes \dotsm \otimes u_\ell) = \kD^g( e^{(n+b)\Upsilon}u_0 \otimes e^{-a_1\Upsilon}u_1 \otimes \dotsm \otimes e^{-a_\ell \Upsilon}u_\ell ) .
\end{equation*}
In particular, if $D$ is also formally self-adjoint, then
\begin{equation*}
 n+b = -a_1 = \dotsm = -a_\ell ,
\end{equation*}
and hence $a_1 = -(n+h)/(\ell+1)$.

Naturality and conformal covariance imply that $U_{\sC}$ is a geometric cone:

\begin{proof}[Proof of Lemma~\ref{lem:geometric-cone}]
 Let $(M,g)$ be a Riemannian $n$-manifold.
 Fix $u \in U_{\sC}^g$ and $C \in \sC$.
 First, since $C$ is linear, $cu \in U_{\sC}^g$.
 Second, conformal invariance implies that if $\hg = e^{2\Upsilon}g$, then
 \begin{equation*}
  C^{\hg}\left( e^{-\frac{n-2k}{r}\Upsilon}u^{\otimes(r(C)-1)} \right) = e^{b(C)\Upsilon}C^g\bigl( u^{\otimes(r(C)-1)} \bigr) > 0 .
 \end{equation*}
 Thus $U_{\sC}^{\hg} = e^{-\frac{n-2k}{r}\Upsilon}U_{\sC}^g$.
 Third, naturality implies that if $f \colon \hM \to M$ is a local diffeomorphism, then
 \begin{equation*}
   C^{f^\ast g}\left( f^\ast u^{\otimes(r(C)-1)} \right) = f^\ast\left( C^g\bigl( u^{\otimes(r(C)-1)} \bigr) \right) > 0 .
 \end{equation*}
 Thus $f^\ast(U_{\sC}^g) \subseteq U_{\sC}^{f^\ast g}$.
\end{proof}

Naturality and conformal covariance also imply that PDE~\eqref{eqn:pde} is $\Conf(M,\kc)$-invariant:

\begin{proof}[Proof of Lemma~\ref{lem:action}]
 Let $(M,\kc)$ be a conformal $n$-manifold and pick $g \in \kc$.
 Let $\Phi \in \Conf(M,\kc)$.
 Then $\Phi^\ast g = \lv J_\Phi \rv^{2/n}g$.
 It follows that $u \cdot \Phi$ defines a right action of $\Conf(M,\kc)$ on $C^\infty(M)$.
 Since $\Phi$ is a diffeomorphism, Lemma~\ref{lem:geometric-cone} yields
 \begin{equation*}
  u \in U_{\sC}^g \Longleftrightarrow \Phi^\ast u \in U_{\sC}^{\Phi^\ast g} = \lv J_\Phi \rv^{w/n}U_{\sC}^g .
 \end{equation*}
 Thus $u \cdot \Phi \in U_{\sC}^g$, and hence $U_{\sC}^g$ is $\Conf(M,\kc)$-invariant.
 
 Suppose now that $u \in U_{\sC}^g$ is such that
 \begin{equation*}
  D^g\bigl(u^{\otimes(r-1)}\bigr) = \lambda \lv u \rv^{\frac{(r-2)n+4k}{n-2k}}u .
 \end{equation*}
 On the one hand, the naturality and conformal covariance of $D$ imply that
 \begin{align*}
  \Phi^\ast \left( D^g\bigl( u^{\otimes(r-1)} \bigr) \right) & = D^{\Phi^\ast g}\bigl( (\Phi^\ast u)^{\otimes(r-1)} \bigr) \\
   & = \lv J_\Phi \rv^{-\frac{(r-1)n+2k}{rn}} D^g \bigl( (u \cdot \Phi)^{\otimes(r-1)} \bigr) .
 \end{align*}
 On the other hand, direct computation gives
 \begin{equation*}
  \Phi^\ast \left( \lv u \rv^{\frac{(r-2)n+4k}{n-2k}}u \right) = \lv J_\Phi \rv^{-\frac{(r-1)n+2k}{rn}} \lv u \cdot \Phi \rv^{\frac{(r-2)n+4k}{n-2k}} (u \cdot \Phi) .
 \end{equation*}
 The conclusion readily follows.
\end{proof}

One construction of formally self-adjoint, conformally covariant, polydifferential operators is by iteratively taking a (weighted) conformal linearization of a CVI:

\begin{proof}[Proof of Theorem~\ref{thm:cly}]
 Given $r \in \bN$, define $D_j^r \colon C^\infty(M)^{\otimes j} \to C^\infty(M)$, $j \in \bN \cup \{ 0 \}$, by $D_0^r = I$ and
 \begin{equation}
  \label{eqn:noncritical-conformal-definition}
  (D_j^r)^g \bigl( u \otimes v^{\otimes(j-1)} \bigr) := \left. \frac{\partial}{\partial t} \right|_{t=0} e^{\frac{(r-1)n+2k}{r}tu} (D_{j-1}^r)^{e^{2tu}g}\bigl( (e^{-\frac{n-2k}{r}tu}v)^{\otimes(j-1)} \bigr)
 \end{equation}
 when $j \in \bN$.
 By induction, $D_j^r$ is formally self-adjoint~\cite{CaseLinYuan2018b}*{Proposition~4.3}.
 The homogeneity of $I$ implies that
 \begin{equation}
  \label{eqn:apply-one}
  D_j^r\bigl(1 \otimes v^{\otimes(j-1)} \bigr) = \frac{(n-2k)(r-j)}{r}D_{j-1}^r\bigl( v^{\otimes(j-1)} \bigr)
 \end{equation}
 for all $v \in C^\infty(M)$~\cite{CaseLinYuan2018b}*{Lemma~4.2}.
 In particular, $D_{2k}^{2k}\bigl(1 \otimes v^{\otimes(2k-1)}\bigr)=0$.
 Considering possible contractions of~\eqref{eqn:building-blocks} implies that $D_{2k}^{2k}=0$~\cite{CaseLinYuan2018b}*{Lemma~3.6(b)}.
 Let $r \in \bN$ be the minimal integer such that $D_r^r=0$.
 Then $D := \frac{1}{(r-1)!}D_{r-1}^r$ is a formally self-adjoint polydifferential operator of rank $r$ that, by Equation~\eqref{eqn:noncritical-conformal-definition}, is conformally covariant.
 Equation~\eqref{eqn:apply-one} implies that $D$ satisfies Condition~\eqref{eqn:associated-operator}.
 
 Suppose now that $\cD$ is a formally self-adjoint, conformally covariant, polydifferential operator of rank $r$ and homogeneity $-2k$ on $n$-manifolds that satisfies Condition~\eqref{eqn:associated-operator}.
 A straightforward induction argument yields
 \begin{equation*}
  D_j^r\bigl( v^{\otimes j} \bigr) = \frac{(r-1)!}{(r-1-j)!}\left(\frac{r}{n-2k}\right)^{r-1-j}\cD \bigl( v^{\otimes j} \otimes 1^{\otimes (r-1-j)} \bigr)
 \end{equation*}
 for all integers $0 \leq j \leq r-1$.
 In particular, $\cD = \frac{1}{(r-1)!}D_{r-1}^r$.
\end{proof}

%% file: ambient.tex
\section{Weyl operators}
\label{sec:ambient}

In this section we describe some formally self-adjoint, conformally covariant, polydifferential operators that are known to be Weyl operators.
We also indicate how to compute such operators at Einstein manifolds.
Doing so requires a quick summary of the Fefferman--Graham ambient space~\cite{FeffermanGraham2012}.

Let $(M,\kc)$ be an conformal $n$-manifold.
Consider the principal $\bR_+$-bundle
\begin{equation*}
 \mG := \left\{ (x,g_x) \suchthat x\in M, g \in \kc \right\}
\end{equation*}
with projection $\pi(x,g_x) := x$ and action $\delta_s(x,g_x) := (x,s^2g_x)$, called the \defn{dilation}.
The tautological symmetric $2$-tensor on $\mG$ is
\begin{equation*}
 \boldsymbol{g}(X,Y) := g_x(\pi_\ast X,\pi_\ast Y)
\end{equation*}
for $X,Y \in T_{(x,g_x)}\mG$.
Extend the dilation $\delta_s$ to $\mG \times \bR$ to act trivially on $\bR$.

A \defn{pre-ambient space} for $(M,\kc)$ is a pair $(\cmG,\cg)$ of a dilation-invariant open set $\cmG \subseteq \mG \times \bR$ containing $\mG \times \{ 0 \}$ and a pseudo-Riemannian metric $\cg$ that is homogeneous of degree $2$ with respect to dilations and satisfies $\iota^\ast\cg = \boldsymbol{g}$, where $\iota \colon \mG \hookrightarrow \mG \times \{ 0 \}$ is the standard inclusion.
Denote by $\rho$ the standard coordinate on $\bR$.
We say that $(\cmG,\cg)$ is an \defn{ambient space} if additionally
\begin{enumerate}
 \item if $n = 2$ or $n$ is odd, then $\Ric(\cg) = O(\rho^\infty)$ along $\mG \times \{ 0 \}$; and
 \item if $n \geq 4$ is even, then $\Ric(\cg) = O^+(\rho^{n/2-1})$.
\end{enumerate}
Two ambient spaces $(\cmG_1,\cg_1)$ and $(\cmG_2,\cg_2)$ for $(M,\kc)$ are \defn{ambient equivalent} if, after shrinking $\cmG_1$ and $\cmG_2$ if necessary, there is a dilation-equivariant diffeomorphism $\Phi \colon \cmG_1 \to \cmG_2$ such that $\Phi\rv_{\mG \times \{0\}} = \Id$ and
\begin{enumerate}
 \item if $n$ is odd, then $\cg_1 - \Phi^\ast\cg_2 = O(\rho^\infty)$; and
 \item if $n$ is even, then $\cg_1 - \Phi^\ast\cg_2 = O^+(\rho^{n/2})$.
\end{enumerate}
Here $O^+(\rho^s)$ consists of those symmetric two-tensors $T \in O(\rho^s)$ on $\cmG$ such that for each $(x,g_x) \in \mG$, there is an $S \in S^2T_x^\ast M$ such that $\tr_{g_x}S = 0$ and $\iota^\ast(\rho^{-s}T) = \pi^\ast S$.
Fefferman and Graham~\cite{FeffermanGraham2012}*{Theorem~2.3} showed that to each conformal $n$-manifold one can associate a unique, up to ambient equivalence, ambient space.
Moreover, one can always assume that $(\cmG,\cg)$ is \defn{straight} and \defn{in normal form};
i.e.
\begin{equation*}
 \cg = 2\rho \, dt^2 + 2t \, dt \, d\rho + t^2g_\rho 
\end{equation*}
for a one-parameter family of metrics $g_\rho$ with $g_0 \in \kc$.

Let $(\cmG,\cg)$ be an ambient space for a conformal $n$-manifold $(M,\kc)$ and let $X$ denote the infinitesimal generator of dilation.
Given $w \in \bR$, let
\begin{equation*}
 \cmE[w] := \left\{ \cu \in C^\infty(\cmG) \suchthat X\cu = w\cu \right\}
\end{equation*}
denote the space of homogeneous functions of weight $w$ on $\cmG$.
The space of \defn{conformal densities} of weight $w \in \bR$ on $(M,\kc)$ is
\begin{equation*}
 \mE[w] := \left\{ \iota^\ast\cu \suchthat \cu \in \cmE[w] \right\} = \left\{ u \in C^\infty(\mG) \suchthat \text{$\delta_s^\ast u = s^wu$ for all $s \in \bR_+$} \right\} .
\end{equation*}
Given $\cu \in \cmE[w]$ and $g \in \kc$, we define $\cu\rv_{(M,g)} \in C^\infty(M)$ by
\begin{equation*}
 \cu\rv_{(M,g)}(x) := \cu(x,g_x,0)
\end{equation*}
and call $\cu$ a \defn{homogeneous extension} of $\cu\rv_{(M,g)}$.
Note that $\cu\rv_{(M,g)}$ depends only on $\iota^\ast\cu$ and $g$.
Moreover, if $\hg = e^{2\Upsilon}g$, then
\begin{equation*}
 \cu\rv_{(M,\hg)} = e^{w\Upsilon}\cu\rv_{(M,g)} .
\end{equation*}
Conversely, if $u \in C^\infty(M)$ and $g \in \kc$ are given, then
\begin{equation*}
 \cu\bigl(x,(e^{2\Upsilon} g)_x,\rho\bigr) := e^{w\Upsilon(x)}u(x)
\end{equation*}
defines an element $\cu \in \cmE[w]$ such that $\cu\rv_{(M,g)} = u$.
In particular, homogeneous extensions always exist.

Let $\cD \colon \cmE[a_1] \otimes \dotsm \otimes \cmE[a_\ell] \to \cmE[b]$ be a natural polydifferential operator on the ambient space of a conformal $n$-manifold.
We say that $\cD$ is \defn{tangential} if
\begin{equation*}
 \iota^\ast \left( \cD( \cu_1 \otimes \dotsm \otimes \cu_\ell ) \right)
\end{equation*}
depends only on $\iota^\ast\cu_1, \dotsc, \iota^\ast\cu_\ell$, and on $\cg$ mod $O^+(\rho^{n/2})$ if $n$ is even.
In this case,
\begin{equation}
 \label{eqn:induced-operator}
 D^g(u_1 \otimes \dotsm \otimes u_\ell) := \cD( \cu_1 \otimes \dotsm \otimes \cu_\ell ) \bigl|_{(M,g)} ,
\end{equation}
where $\cu_j \in \cmE[a_j]$ are homogeneous extensions of $u_j$, defines a conformally covariant\footnote{
 The domain and codomain of $\cD$ imply that $D$ transforms as in Equation~\eqref{eqn:conformal-covariant-bidegree}.
}
polydifferential operator of rank $\ell+1$.
A \defn{Weyl operator} is a conformally covariant polydifferential operator that can be defined holographically in this way.
If $\cI \in \cmE[w]$ is a natural scalar Riemannian invariant, then Equation~\eqref{eqn:induced-operator} determines a scalar conformal invariant $\mI$ of weight $w$~\cite{FeffermanGraham2012}.

Graham, Jenne, Mason, and Sparling~\cite{GJMS1992} made two important observations for the construction of Weyl operators.
First, $Q := \cg(X,X) \in \cmE[2]$ is a defining function for $\iota(\mG) \subset \cmG$.
In particular, if $\cu,\cv \in \cmE[w]$, then $\iota^\ast\cu = \iota^\ast\cv$ if and only if $\cu - \cv = Q\cf$ for some $\cf \in \cmE[w-2]$.
Second, regarding $Q$ as a multiplication operator and $X$ as a derivation on $\cmG$, one has the commutator identity\footnote{
 Recall our convention $\cDelta \geq 0$; this is responsible for a sign discrepancy with~\cite{GJMS1992}.
}~\cite{GJMS1992}*{Equation~(1.8)}
\begin{equation}
 \label{eqn:sl2-algebra}
 [\cDelta^k,Q] = -2k\cDelta^{k-1}(2X + n + 4 - 2k) .
\end{equation}
Suppose $\cD \colon \cmE[a_1] \otimes \dotsm \otimes \cmE[a_\ell] \to \cmE[b]$ is a natural polydifferential operator that is expressed as a linear combination of compositions of the ambient Laplacian and multiplication operators.
Equation~\eqref{eqn:sl2-algebra} allows one to check if each commutator
\begin{multline}
 \label{eqn:commutator-defn}
 [\cD,Q]_j(\cu_1 \otimes \dotsm \otimes \cu_\ell) := \cD( \cu_1 \otimes \dotsm \otimes \cu_{j-1} \otimes Q\cu_j \otimes \cu_{j+1} \otimes \dotsm \otimes \cu_{\ell} ) \\
  - Q\cD( \cu_1 \otimes \dotsm \otimes \cu_\ell )
\end{multline}
vanishes on $\cmE[a_1] \otimes \dotsm \otimes \cmE[a_{j-1}] \otimes \cmE[a_j-2] \otimes \cmE[a_{j+1}] \otimes \dotsm \otimes \cmE[a_\ell]$.
If so, and if $\cD$ depends only on $\cg$ mod $O^+(\rho^{n/2})$ if $n$ is even, then $\cD$ is tangential.
This is the approach taken in each of the examples of this section.

Graham et al.~\cite{GJMS1992} gave the first general construction of formally self-adjoint, conformally covariant, differential operators:

\begin{example}[GJMS operators~\cite{GJMS1992}]
 \label{ex:gjms-operator}
 Fix $k,n \in \bN$;
 if $n$ is even, then assume additionally that $k \leq n/2$.
 Equation~\eqref{eqn:sl2-algebra} implies that $\cDelta^k$ is tangential on $\cmE\bigl[-\frac{n-2k}{2}\bigr]$.
 The induced operator $L_{2k}$ has rank $2$ and homogeneity $-2k$.
 
 There are four proofs that $L_{2k}$ is formally self-adjoint:
 via scattering theory~\cite{GrahamZworski2003},
 via the renormalized Dirichlet form of a second-order operator~\cite{FeffermanGraham2002},
 via Juhl's formula~\cites{Juhl2009,FeffermanGraham2013},
 and via the renormalized Dirichlet form of an operator of order $2k$~\cite{CaseYan2024}.
 The first two approaches use the realization of $L_{2k}u$ as the obstruction to the existence of a formally harmonic homogeneous extension $\cu \in \cmE\bigl[-\frac{n-2k}{2}\bigr]$ of $u$~\cite{GJMS1992}, while the last two use the construction of $L_{2k}$ via a tangential operator.
 \qed
\end{example}

There are two GJMS operators that have attracted particular attention.
The \defn{conformal Laplacian} is the second-order GJMS operator
\begin{equation*}
 L_2 := \Delta + \frac{n-2}{2}J .
\end{equation*}
It features heavily in the solution of the Yamabe Problem~\cite{LeeParker1987}.
The \defn{Paneitz operator}~\cite{Paneitz1983} is the fourth-order GJMS operator
\begin{align*}
 L_4 & := \Delta^2 - \delta\bigl(4P-(n-2)Jg\bigr)d + \frac{n-4}{2}Q_4 , \\
 Q_4 & := \Delta J - 2\lv P \rv^2 + \frac{n}{2}J^2 .
\end{align*}
It is especially important in studying compact Riemannian four-manifolds~\cite{Chang2005}.
We call $Q_4$ the \defn{fourth-order $Q$-curvature}.
In these formulas,
\begin{align*}
 P & := \frac{1}{n-2}\left( \Ric - J g \right) , & J & := \frac{R}{2(n-1)} ,
\end{align*}
are the \defn{Schouten tensor} and its trace, respectively, and
\begin{equation*}
 \delta\omega := -\nabla^a\omega_a
\end{equation*}
is the $L^2$-adjoint of the exterior derivative on functions.

The (noncritical) \defn{$Q$-curvature of order $2k$} on $n$-manifolds, $n > 2k$, is the scalar Riemannian invariant defined by $L_{2k}(1) = \frac{n-2k}{2}Q_{2k}$.
While $L_{2k}$ has minimal rank, there are other formally self-adjoint, conformally covariant, polydifferential operators associated to $Q_{2k}$ in the sense of Equation~\eqref{eqn:associated-operator}:

\begin{example}[GJMS operators of rank $4$]
 \label{ex:nonminimal-gjms-operator}
 \pushQED{\qed}
 Fix $k,n \in \bN$ such that $2k < n$.
 Comparison with Example~\ref{ex:gjms-operator} implies that
 \begin{equation*}
  \cD_{2k}( \cu^{\otimes 3} ) := \frac{(n-2k)^2}{32}\cu\cDelta^k\cu^2
 \end{equation*}
 determines a tangential operator on $\cmE\bigl[-\frac{n-2k}{4}\bigr]^{\otimes 3}$ by polarization, and the induced conformally covariant operator $D_{2k}$ satisfies
 \begin{equation*}
  D_{2k}\bigl( u^{\otimes 3} \bigr) = \frac{(n-2k)^2}{32} u L_{2k}\bigl(u^{2}\bigr) ,
 \end{equation*}
 is formally self-adjoint, and has rank $4$.
 Our normalizations are such that
 \begin{equation*}
  D_{2k}(1^{\otimes3}) = \left(\frac{n-2k}{4}\right)^3Q_{2k} . \qedhere
 \end{equation*}
 \popQED
\end{example}

Example~\ref{ex:gjms-operator} already indicates the utility of the ambient Laplacian for constructing Weyl operators.
Another benefit of the ambient Laplacian is that it leads to a simple proof (cf.\ \citelist{ \cite{FeffermanGraham2012}*{Proposition~7.9} \cite{Gover2006q}*{Theorem~1.2} }) of the factorization of the GJMS operators at Einstein manifolds.
This is because, as first observed by Matsumoto~\cite{Matsumoto2013}*{Lemma~4.1} in the context of the Lichnerowicz Laplacian, the Laplacian of the ambient space~\cite{FeffermanGraham2012}*{Equation~(7.12)}
\begin{equation}
 \label{eqn:einstein-ambient}
 \begin{split}
  \cmG & := (0,\infty)_t \times M \times (-\varepsilon,\varepsilon)_\rho , \\
  \cg & := 2\rho \, dt^2 + 2t \, dt \, d\rho + t^2\left(1+\lambda\rho/2\right)^2g ,
 \end{split}
\end{equation}
of an Einstein $n$-manifold with $\Ric = (n-1)\lambda g$ takes a particularly simple form.

\begin{proposition}
 \label{prop:factorization}
 Fix $k,n \in \bN$;
 if $n$ is even, then assume additionally that $k \leq n/2$.
 Let $(M,g)$ be an Einstein $n$-manifold with $\Ric = (n-1)\lambda g$.
 Then
 \begin{equation}
  \label{eqn:factorization}
  L_{2k} = \prod_{j=0}^{k-1} \left( \Delta + \frac{(n-2j-2)(n+2j)}{4}\lambda \right) .
 \end{equation}
\end{proposition}

\begin{proof}
 Let $(\cmG,\cg)$ be as in Display~\eqref{eqn:einstein-ambient}.
 Set $\tau := t(1+\lambda\rho/2)$.
 Direct computation implies that if $u \in C^\infty(M)$ and $w \in \bR$, then
 \begin{equation}
  \label{eqn:einstein-laplacian}
  \cDelta(\tau^w\pi^\ast u) = \tau^{w-2}\pi^\ast ( \Delta u - w(n+w-1)\lambda u ) ;
 \end{equation}
 see~\cite{CaseLinYuan2022or}*{Lemma~5.1}.
 Set $w := -(n-2k)/2$.
 Iterating Equation~\eqref{eqn:einstein-laplacian} yields
 \begin{align*}
  \cDelta^k(\tau^w\pi^\ast u) & = \tau^{w-2k}\pi^\ast (Du) , \\
  D & := \prod_{j=0}^{k-1} \left( \Delta + \frac{(n-2k+4j)(n+2k-4j-2)}{4}\lambda \right) .
 \end{align*}
 Since $\cDelta^k$ is tangential on $\cmE[w]$, we see that $L_{2k}=D$.
 The final conclusion follows by reindexing the product.
\end{proof}

The factorization~\eqref{eqn:factorization} plays an important role in studying Poincar\'e--Einstein manifolds;
e.g.\ in computing renormalized volumes~\cite{ChangQingYang2006} or deriving interior energy identities for the fractional GJMS operators~\cite{CaseChang2013}.
We conjecture that Equation~\eqref{eqn:factorization} \emph{characterizes} the GJMS operators:

\begin{conjecture}
 \label{conj:gjms-einstein}
 Let $D$ be a formally self-adjoint, conformally covariant, polydifferential operator of rank $2$ and homogeneity $-2k$ on $n$-manifolds, where $n \geq 2k$.
 If $D$ factors as in Equation~\eqref{eqn:factorization} on Einstein manifolds, then $D = L_{2k}$.
\end{conjecture}

A related conjecture~\citelist{ \cite{Branson2005}*{Remark~4} \cite{AIM2003-Conference}*{Problem~11} } asserts that $L_{2k}$ is the unique conformally covariant operator with leading-order term $\Delta^k$ that can be expressed as a linear combination of complete contractions of covariant derivatives of the Ricci tensor and the input function.

Conjecture~\ref{conj:gjms-einstein} is not trivial, as $L_{2k}$ is not the unique formally self-adjoint, conformally covariant, polydifferential operator of rank $2$, homogeneity $-2k$, and leading-order term $\Delta^k$.
For example, adding a nontrivial scalar conformal invariant of weight $-2k$ (e.g.\ a complete contraction of $W^{\otimes k}$) produces another such operator.
The following building blocks of Case and Yan~\cite{CaseYan2024} allow one to construct perturbations of $L_{2k}$ through higher-order terms:

\begin{lemma}
 \label{lem:case-yan}
 Let $(\cmG,\cg)$ be a straight and normal ambient space for a conformal $n$-manifold $(M,\kc)$ and let $\cf \in \cmE[w']$.
 Let $k \in \bN$ and $w \in \bR$.
 Define
 \begin{equation}
  \label{eqn:case-yan-operator}
  \cD_{2k,w,\cf}(\cu) := \sum_{a+b=k} \frac{k!}{a!b!}\frac{\Gamma\bigl(\frac{n-2k}{2}+w+a\bigr)\Gamma\bigl(b-w-w'-\frac{n-2k}{2}\bigr)}{\Gamma\bigl(\frac{n-2k}{2}+w\bigr) \Gamma\bigl(-w-w'-\frac{n-2k}{2}\bigr)}\cDelta^a(\cf\cDelta^b\cu) 
 \end{equation}
 on $\cmE[w]$.
 Then $\iota^\ast\cD_{2k,w,\cf}(\cu)$ depends only on $\iota^\ast\cu$, $\cf$, and $\cg$.
 If also $w = -\frac{n+w'-2k}{2}$, then the operator $D_{2k,w,\cf}^g \colon C^\infty(M) \to C^\infty(M)$ defined by
 \begin{equation*}
  D_{2k,w,\cf}^g(u) := \left. \cD_{2k,w,\cf}(t^w\pi^\ast u) \right|_{(M,g)}
 \end{equation*}
 is formally self-adjoint.
\end{lemma}

\begin{proof}
  The first claim~\cite{CaseYan2024}*{Lemma~2.4} follows from Equation~\eqref{eqn:sl2-algebra}.
  The second claim~\cite{CaseYan2024}*{Lemma~4.1} follows by considering the logarithmic part of the polyhomogeneous expansion, in $\varepsilon$, of
 \begin{equation*}
  \int_{\{ r>\varepsilon \}} \left. \left( \cv \cD_{2k,w,\cf}(\cu) \right) \right|_{\mathring{X}} \dvol_{g_+} ,
 \end{equation*}
 where $(\mathring{X},g_+)$ is the Poincar\'e manifold determined by $(\cmG,\cg)$.
\end{proof}

\begin{remark}
 Chern and Yan~\cite{ChernYan2024}*{Theorem~1.2} gave an alternative proof of the formal self-adjointness of $D_{2k,w,\cf}^g$ by deriving a Juhl-type formula.
\end{remark}

Lemma~\ref{lem:case-yan} does \emph{not} give rise to conformally covariant differential operators on $M$, as the operator $\iota^\ast\cu \mapsto \iota^\ast\bigl(\cD_{2k,w,\cf}(\cu)\bigr)$ depends on $(\cmG,\cg)$ via its dependence on $\cf$.
However, choosing $\cf$ to be an ambient scalar Riemannian invariant which is independent of the ambiguity of $\cg$ \emph{does} produce a conformally covariant polydifferential operator of rank $2$:

\begin{example}[Polydifferential operators of rank two~\cite{CaseYan2024}]
 \label{ex:linear-operator}
 Fix $k,n \in \bN$ and let $\cI \in \cmE[-2\ell]$ be an ambient scalar Riemannian invariant;
 if $n$ is even, then assume additionally\footnote{
  If $\ell \geq 1$, then one need only assume that $k + \ell \leq n/2+1$.
  This ensures that the ambiguity of the ambient metric does not affect the induced operator~\cite{CaseYan2024}*{Corollary~2.5}.
 }
 that $k + \ell \leq n/2$.
 Lemma~\ref{lem:case-yan} implies that
 \begin{equation*}
  \cL_{2k,\cI} := \cD_{2k,-\frac{n-2k-2\ell}{2},\cI} \colon \cmE\left[ -\frac{n-2k-2\ell}{2} \right] \to \cmE\left[ -\frac{n+2k+2\ell}{2} \right]
 \end{equation*}
 induces a formally self-adjoint, conformally covariant, polydifferential operator of rank $2$ and homogeneity $-2k-2\ell$, which we denote by $L_{2k,\mI}$.
 \qed
\end{example}

If $k=0$ or $\ell=0$, then the operator $L_{2k,\mI}$ recovers multiplication by the scalar conformal invariant $\mI$ or a multiple of the GJMS operator $L_{2k}$, respectively.

Case, Lin, and Yuan~\cite{CaseLinYuan2022or} gave the first general construction of a family of conformally covariant polydifferential operators of rank $3$ via the ambient metric, generalizing operators found on the conformal $n$-sphere by Ovsienko and Redou~\cite{OvsienkoRedou2003} and Clerc~\cites{Clerc2016,Clerc2017}.
Following Case and Yan~\cite{CaseYan2024}, we use Lemma~\ref{lem:case-yan} to present the case when these operators are formally self-adjoint.

\begin{example}[Ovsienko--Redou operators~\cite{CaseLinYuan2022or}]
 \label{ex:ovsienko-redou-operator}
 Fix $k,n \in \bN$;
 if $n$ is even, then assume additionally that $k \leq n/2$.
 Consider the operator
 \begin{align}
  \label{eqn:ovsienko-redou-form} \cD_{2k}(\cu \otimes \cv) & := \sum_{a+b+c=k} A_{a,b,c} \cDelta^a \left( \bigl(\cDelta^b\cu\bigr) \bigl(\cDelta^c\cv\bigr) \right) , \\
  A_{a,b,c} & := \frac{k!}{a!b!c!} \frac{\Gamma\bigl( \frac{n-2k}{6}+a+b \bigr) \Gamma\bigl( \frac{n-2k}{6}+a+c \notag \bigr) \Gamma\bigl( \frac{n-2k}{6}+b+c \bigr)}{ \Gamma\bigl(\frac{n-2k}{6}\bigr) \Gamma\bigl(\frac{n+4k}{6}\bigr)^2 } ,
 \end{align}
 on $\cmE\bigl[-\frac{n-2k}{3}\bigr]^{\otimes2}$.
 Clearly $\cD_{2k}$ is symmetric.
 Direct computation yields
 \begin{equation}
  \label{eqn:or-to-cy}
  \cD_{2k}(\cu \otimes \cv) = \sum_{a+b=k} \frac{k!}{a!b!}\frac{\Gamma\bigl( \frac{n-2k}{6}+a \bigr) \Gamma\bigl( \frac{n-2k}{6}+b \bigr)^2}{\Gamma \bigl( \frac{n-2k}{6} \bigr) \Gamma\bigl( \frac{n+4k}{6} \bigr)^2} \cD_{2a,-\frac{n-2k}{3},\cDelta^b\cu}(\cv) .
 \end{equation}
 We conclude from Lemma~\ref{lem:case-yan} that $\cD_{2k}$ induces a formally self-adjoint, conformally covariant, polydifferential operator $D_{2k}$ of rank $3$ and homogeneity $-2k$.
 \qed
\end{example}

The argument in the proof of Proposition~\ref{prop:factorization} can be used to compute $L_{2k,\mI}$ and $D_{2k}$ at Einstein manifolds.
However, the resulting formula~\cite{CaseLinYuan2018b}*{Theorem~1.4} is not manifestly formally self-adjoint without further algebraic manipulation, and it is not presently clear what is the ``best'' way to organize terms.

Unlike the rank $2$ case~\cite{Branson1995}, there are conformally covariant bidifferential operators on the round $n$-sphere that are not formally self-adjoint.
Indeed, Ovsienko and Redou~\cite{OvsienkoRedou2003} classified these operators in the generic case, and Clerc~\cites{Clerc2016,Clerc2017} classified them in general.
Case, Lin, and Yuan showed~\cite{CaseLinYuan2022or} that these operators all have curved analogues on conformal manifolds when the homogeneity is at least $-n$.
We use Lemma~\ref{lem:case-yan} to discuss a particular family of such operators that arise when checking the Frank--Lieb Property for the Ovsienko--Redou operators:

\begin{example}[More bidifferential operators~\cite{CaseLinYuan2022or}]
 \label{ex:non-fsa}
 Fix $k,n \in \bN$;
 if $n$ is even, then assume additionally that $k \leq n/2$.
 Consider the operator
 \begin{align*}
  \cD_{2k}^\prime( \cu \otimes \cv) & := \sum_{a+b+c=k} B_{a,b,c} \cDelta^a \left( \bigl( \cDelta^b\cu \bigr) \bigl( \cDelta^c\cv \bigr) \right) , \\
  B_{a,b,c} & := \frac{k!}{a!b!c!} \frac{ \Gamma\bigl( \frac{n-2k+4}{6} + a + b \bigr) \Gamma\bigl( \frac{n-2k-2}{6} + a + c \bigr) \Gamma\bigl( \frac{n-2k-2}{6} + b + c \bigr)}{\Gamma\bigl( \frac{n-2k-2}{6} \bigr) \Gamma\bigl( \frac{n+4k+4}{6} \bigr)^2} ,
 \end{align*}
 on $\cmE\bigl[-\frac{n-2k+1}{3}\bigr] \otimes \cmE\bigl[-\frac{n-2k-2}{3}\bigr]$.
 Direct computation yields
 \begin{align*}
  \cD_{2k}^\prime(\cu \otimes \cv) & = \sum_{a+b=k} B_{a,b}' D_{2a,-\frac{n-2k+1}{3},\cDelta^b\cv}(\cu) \\
  & = \sum_{a+b=k} B_{a,b}'' D_{2a,-\frac{n-2k-2}{3},\cDelta^b\cu}(\cv) ,
 \end{align*}
 where
 \begin{align*}
  B_{a,b}' & := \frac{k!}{a!b!} \frac{ \Gamma\bigl(\frac{n-2k+4}{6}+a\bigr) \Gamma\bigl(\frac{n-2k-2}{6}+b\bigr)^2 }{ \Gamma\bigl(\frac{n-2k-2}{6}\bigr) \Gamma\bigl(\frac{n+4k+4}{6}\bigr)^2 } , \\
  B_{a,b}'' & := \frac{k!}{a!b!} \frac{ \Gamma\bigl(\frac{n-2k-2}{6}+a\bigr) \Gamma\bigl(\frac{n-2k-2}{6}+b\bigr) \Gamma\bigl(\frac{n-2k+4}{6}+b\bigr) }{ \Gamma\bigl(\frac{n-2k-2}{6}\bigr) \Gamma\bigl(\frac{n+4k+4}{6}\bigr)^2 } .
 \end{align*}
 Lemma~\ref{lem:case-yan} implies that $\cD_{2k}^\prime$ induces a conformally covariant polydifferential operator $D_{2k}^\prime$ of rank $3$ and homogeneity $-2k$.
 \qed
\end{example}

We conclude by discussing the operator (cf.\ \citelist{ \cite{CaseWang2016s}*{Lemma~3.1} \cite{Case2019fl}*{Remark~2.2} })
\begin{multline}
 \label{eqn:sigma2-operator}
 D(u^{\otimes 3}) := -\frac{1}{2}\delta\bigl( \lv\nabla u\rv^2 \, du \bigr) + \frac{n-4}{16}\left( u\Delta\lv\nabla u\rv^2 + \delta\bigl( (\Delta u^2)\,du \bigr) \right) \\
  + \frac{1}{2}\left( \frac{n-4}{4} \right)^2 u\delta\bigl( T_1(\nabla u^2) \bigr) + \left( \frac{n-4}{4} \right)^3\sigma_2u^3
\end{multline}
associated to the $\sigma_2$-curvature
\begin{equation*}
 \sigma_2 := \frac{1}{2}\left( J^2 - \lv P \rv^2 \right)
\end{equation*}
by Theorem~\ref{thm:cly}, where $T_1 := Jg - P$.
The operator $D$ is a formally self-adjoint, conformally covariant, polydifferential operator of rank $4$ and homogeneity $-4$ on $n$-manifolds.
It and a related Dirichlet form have been used to study sharp fully nonlinear Sobolev inequalities in Euclidean space~\cites{CaseWang2019,CaseWang2016s,Gonzalez2006r}.
We give an alternative realization of $D$ as a Weyl operator (cf.\ \cite{CaseLinYuan2018b}):

\begin{example}[The $\sigma_2$-operator~\cite{CaseLinYuan2018b}]
 \label{ex:sigma2-operator}
 Fix $n \in \bN$.
 Consider the operator
 \begin{equation*}
  \cD(\cu^{\otimes 3}) := -\frac{n-2}{16}\cu\cDelta^2\cu^2 + \frac{n}{96}\left( \cDelta^2\cu^3 + 3\cu^2\cDelta^2\cu \right) + \frac{n-4}{16}\left( \cu(\cDelta\cu)^2 + \cDelta(\cu^2\cDelta\cu) \right)
 \end{equation*}
 on $\cmE\bigl[ -\frac{n-4}{4} \bigr]^{\otimes 3}$.
 On the one hand, the identities
 \begin{align}
  \label{eqn:delta-to-Delta} 2\cdelta(\cu\,\cd\cv) & = \cDelta(\cu\cv) + \cu\cDelta\cv - \cv\cDelta\cu , \\
  \label{eqn:Delta-on-three} \cDelta(\cu\cv\cw) + \cu\cv\cDelta\cw + \cu\cw\cDelta\cv + \cv\cw\cDelta\cu & = \cu\cDelta(\cv\cw) + \cv\cDelta(\cu\cw) + \cw\cDelta(\cu\cv) ,
 \end{align}
 imply that
 \begin{equation*}
  \cD(\cu^{\otimes3}) = -\frac{1}{2}\cdelta\left( \lv\cnabla\cu\rv^2\,\cd\cu \right) + \frac{n-4}{16}\left( \cu\cDelta\lv\cnabla\cu\rv^2 + \cdelta\bigl( (\cDelta\cu^2) \, \cd\cu \bigr) \right) .
 \end{equation*}
 Hence $D$ is induced by $\cD$~\cite{Case2019fl}*{Lemma~2.1}.
 On the other hand, we may write
 \begin{equation}
  \label{eqn:sigma2-operator-to-case-yan}
  \begin{split}
   \cD(\cu \otimes \cv^{\otimes2}) & = -\frac{n-2}{24}\cv\cDelta^2(\cu\cv) + \frac{1}{6(n-4)}\cD_{4,-\frac{n-4}{4},\cv^2}(\cu) \\
    & \quad + \frac{n-4}{6n}\cD_{2,-\frac{n-4}{4},\cv\cDelta\cv}(\cu) - \frac{n-2}{48}\cu\cDelta^2(\cv^2) \\
    & \quad + \frac{n}{48}\cu\cv\cDelta^2\cv + \frac{n-4}{48}\cu(\cDelta\cv)(\cDelta\cw) .
  \end{split}
 \end{equation}
 Lemma~\ref{lem:case-yan} implies that $D$ is a formally self-adjoint, conformally covariant, polydifferential operator of rank $4$ and homogeneity $-4$.
 \qed
\end{example}

Case and Cieslak~\cite{CaseCieslak2025} gave additional examples of conformally covariant polydifferential operators of rank $4$ that are conjecturally formally self-adjoint.

We conclude by discussing the renormalized volume coefficients:

\begin{example}[The $v_k$-operators]
 \label{ex:vk}
 The renormalized volume coefficient $v_k$ is a CVI~\cite{Graham2009}*{Theorem~1.5} that recovers the $\sigma_k$-curvature when $k \leq 2$ or when evaluated at locally conformally flat manifolds~\cite{GrahamJuhl2007}*{Proposition~1}.
 The associated operators have rank $2k$ and homogeneity $-2k$~\cite{CaseLinYuan2018b}*{Proposition~5.5}, and are Weyl operators~\cite{CaseLinYuan2018b}*{Theorem~6.6}.
\end{example}

%% file: frank-lieb.tex
\section{The Frank--Lieb argument}
\label{sec:frank-lieb}

The Frank--Lieb argument~\cites{FrankLieb2012a,FrankLieb2012b} classifies local minimizers of sharp Sobolev inequalities by exploiting conformal covariance and carefully studying the spectral properties of the underlying operator.
In this section we describe the generalization of this argument to polydifferential operators~\cite{Case2019fl}.

Conformal covariance helps produce useful test functions for the second variation at a local minimizer.
To that end, given $w < 0$, we say that $u \in C^\infty(S^n)$ is \defn{$w$-balanced} if
\begin{equation*}
 \int_{S^n} \vec{x} \lv u \rv^{-n/w} \dvol_{d\theta^2} = 0 ,
\end{equation*}
where $\vec{x}$ is the position vector on $S^n$, regarded as the unit sphere in $\bR^{n+1}$ with the induced metric $d\theta^2$ of constant sectional curvature $1$.
In particular, if $u$ is $w$-balanced, then for each $i \in \{ 0, \dotsc, n \}$ one computes that
\begin{equation*}
 \left. \frac{d}{dt} \right|_{t=0} \int_{S^n} \lv (1+tx^i)u \rv^{-n/w} \dvol_{d\theta^2} = 0 .
\end{equation*}
One can use the action of the conformal group of $S^n$ to make a function $w$-balanced:

\begin{lemma}
 \label{lem:balancing}
 Fix $w < 0$ and consider the round $n$-sphere $(S^n,d\theta^2)$.
 For each $u \in C^\infty(S^n)$, there is a $\Phi \in \Conf(S^n,\kc_0)$ such that $u \cdot \Phi$ is $w$-balanced.\footnote{
  The action $u \cdot \Phi$ is defined in terms of $n$ and $w$ by Equation~\eqref{eqn:action}.
 }
\end{lemma}

\begin{proof}
 We need only consider the case $u \not= 0$.
 An application~\cite{FrankLieb2012b}*{Lemma~B.1} of Brouwer's fixed point theorem implies that there is a $\Phi \in \Conf(S^n)$ such that
 \begin{equation*}
  \int_{S^n} (\vec{x} \circ \Phi^{-1} ) \lv u \rv^{-n/w} \dvol_{d\theta^2} = 0 .
 \end{equation*}
 Hence, by change of variables,
 \begin{equation*}
  0 = \int_{S^n} \vec{x} \lv \Phi^\ast u \rv^{-n/w} \lv J_\Phi \rv \dvol_{d\theta^2} = \int_{S^n} \vec{x} \lv u \cdot \Phi \rv^{-n/w} \dvol_{d\theta^2} . \qedhere
 \end{equation*}
\end{proof}

The role of the Frank--Lieb Property is to extract consequences from the nonnegativity of the second variation at a local minimizer:

\begin{proof}[Proof of Theorem~\ref{thm:frank-lieb}]
 Let $u \in U_{\sC}^{d\theta^2} \setminus \{ 0 \}$ be a local minimizer of $\mF^{d\theta^2}$.
 By scaling, we may assume that $\lV u \rV_{r^\ast} = 1$.
 Since $\mF^{d\theta^2}$ is $\Conf(S^n)$-invariant, Lemma~\ref{lem:balancing} implies that we may also assume that $u$ is $-(n-2k)/r$-balanced.
 
 Since $u$ is a local minimizer of $\mF^{d\theta^2}$, we compute that
 \begin{align*}
  0 & \leq \frac{1}{r(r-1)}\left. \frac{d^2}{dt^2} \right|_{t=0} \mF^{d\theta^2}\bigl( (1+tv)u \bigr) \\
   & = \kD^{d\theta^2}( (uv)^{\otimes 2} \otimes u^{\otimes (r-2)}) - \frac{(r-1)n+2k}{(r-1)(n-2k)}\kD^{d\theta^2}(u^{\otimes r})\int_{S^n} v^2\lv u \rv^{\frac{rn}{n-2k}} \dvol_{d\theta^2}
 \end{align*}
 for all $v \in C^\infty(S^n)$ such that
 \begin{equation}
  \label{eqn:volume-preserved}
  \int_{S^n} v\lv u \rv^{\frac{rn}{n-2k}} \dvol_{d\theta^2} = 0 .
 \end{equation}
 Since $u$ is balanced, the functions $v=x^i$ satisfy Equation~\eqref{eqn:volume-preserved}.
 Thus
 \begin{equation*}
  \sum_{i=0}^n \kD^{d\theta^2}\bigl( (x^iu)^{\otimes 2} \otimes u^{\otimes (r-2)}\bigr) \geq \frac{(r-1)n+2k}{(r-1)(n-2k)}\kD^{d\theta^2}(u^{\otimes r}) .
 \end{equation*}
 Combining this with the commutator identity
 \begin{equation*}
  \sum_{i=0}^n \kD^{d\theta^2}( (x^iu)^{\otimes 2} \otimes u^{\otimes (r-2)}) = \kD^{d\theta^2}(u^{\otimes r}) + \sum_{i=0}^n \int_{S^n} x^iu [D^{d\theta^2},x^i]_1(u^{\otimes(r-1)}) \dvol_{d\theta^2}
 \end{equation*}
 yields
 \begin{equation*}
  \sum_{i=0}^n \int_{S^n} x^iu [D^{d\theta^2},x^i](u^{\otimes(r-1)}) \dvol_{d\theta^2} \geq \frac{2rk}{(r-1)(n-2k)}\kD^{d\theta^2}(u^{\otimes r}) .
 \end{equation*}
 Since $(D,\sC)$ satisfies the Frank--Lieb Property, we conclude that $u$ is constant.
\end{proof}

\begin{remark}
 \label{rk:fl-regularity}
 The proof of Theorem~\ref{thm:frank-lieb} only requires that $u$ is a local minimizer of a continuous extension of $\mF$ to a completion of $U_{\sC}^{d\theta^2}$ for which Condition~\eqref{eqn:fl-poincare} remains valid.
 In particular, Theorem~\ref{thm:frank-lieb} establishes regularity of the minimizers for the norm inequality for the embedding $W^{k,2}(S^n) \hookrightarrow L^{2n/(n-2k)}(S^n)$, $n > 2k$, without applying elliptic theory (cf.\ \cite{Yan2022}).
\end{remark}

There are presently two cases where the Frank--Lieb Property is verified~\cite{Case2019fl}:
the GJMS operators and the operator associated to the $\sigma_2$-curvature.
Additionally, the commutator $\sum x^i[D,x^i]_1$ has been computed for the Ovsienko--Redou operators~\cite{CaseLinYuan2022or}.
More generally, this commutator is easily computed for Weyl operators:

\begin{lemma}
 \label{lem:ambient-commutator-computation}
 Let $D$ be a formally self-adjoint, conformally covariant, polydifferential Weyl operator of rank $r$ and homogeneity $-2k$ on $n$-manifolds.
 Then
 \begin{multline*}
  \sum_{i=0}^n x^i[D^{d\theta^2},x^i]_1\bigl(u \otimes v^{\otimes(r-2)}\bigr) \\
  = \sum_{i=0}^n \left. \Bigl( \cx^i\ct \cD\bigl( \cx^i\cu \otimes \cv^{\otimes(r-2)}\bigr) - (\cx^i)^2\cD\bigl( \ct\cu \otimes \cv^{\otimes(r-2)}\bigr) \Bigr) \right|_{(S^n,d\theta^2)} ,
 \end{multline*}
 where $\cD \colon \cmE\bigl[-\frac{n-2k}{r}\bigr]^{\otimes(r-1)} \to \cmE\bigl[-\frac{(r-1)n+2k}{r}\bigr]$ defines $D$ holographically,
 \begin{equation*}
  (\cmG,\cg) = \Bigl( \bR_{\ct} \times \bR_{\cx}^{n+1} , -d\ct^2 + \sum_{i=0}^n (dx^i)^2 \Bigr)
 \end{equation*}
 is the flat ambient space for $(S^n,\kc_0)$, and $\cu \in \cmE\bigl[-\frac{n-2k}{r}-1\bigr]$ and $\cv \in \cmE\bigl[-\frac{n-2k}{r}\bigr]$ are homogeneous extensions of $u$ and $v$, respectively.
\end{lemma}

\begin{proof}
 Note that $\ct \in \cmE[1]$ and $\cx^i \in \cmE[1]$ are homogeneous extensions of the constant function and the standard Cartesian coordinates, respectively, on $(S^n,d\theta^2)$, regarded as $\left\{ (\ct,\cx) \suchthat \ct=1, \lv \cx \rv^2 = 1 \right\} \subset \cmG$.
 Then $\ct\cu, \cx^i\cu \in \cmE\bigl[-\frac{n-2k}{r}\bigr]$ are homogeneous extensions of $u,x^iu$, respectively.
 Therefore
 \begin{align*}
  D^{d\theta^2}\bigl(u \otimes v^{\otimes(r-2)} \bigr) & = \left. \cD \bigl( \ct\cu \otimes \cv^{\otimes(r-2)} \bigr) \right|_{(S^n,d\theta^2)} , \\
  D^{d\theta^2}\bigl( x^iu \otimes v^{\otimes(r-2)} \bigr) & = \left. \cD \bigl( \cx^i\cu \otimes \cv^{(r-2)} \bigr) \right|_{(S^n,d\theta^2)} .
 \end{align*}
 The conclusion readily follows.
\end{proof}

The next lemma allows one to easily compute commutators for Weyl operators expressed in terms of ambient Laplacians.
Our proof simplifies that of Case, Lin, and Yuan~\cite{CaseLinYuan2022or}*{Lemma~4.1}.

\begin{lemma}
 \label{lem:commutator-building-blocks}
 Let $(\cmG,\cg)$ denote the flat ambient space of the standard conformal $n$-sphere.
 Let $k \in \bN$.
 Then
 \begin{equation}
  \label{eqn:commutator-building-block-one-laplacian}
  \sum_{i=0}^n \cx^i\left( \ct\,[\cDelta^k,\cx^i] - \cx^i[\cDelta^k,\ct\,] \right) = -2k\ct X\cDelta^{k-1} + 2kQ(\cnabla\ct\,)\cDelta^{k-1}
 \end{equation}
 as operators on $C^\infty(\cmG)$.
 Additionally, if $\cu,\cv \in C^\infty(\cmG)$, then
 \begin{multline}
  \label{eqn:commutator-building-lbock-two-laplacian}
  \sum_{i=0}^n \cx^i\left( \ct\cDelta^k(\cu(\cnabla\cx^i)\cv) - \cx^i\cDelta^k(\cu(\cnabla\ct\,)\cv) \right) \\
   = \ct\cDelta^k( \cu ( X - k)\cv) - k\ct\cDelta^{k-1}(\cu\cDelta\cv - \cv\cDelta\cu) - Q\cDelta^k(\cu(\cnabla\ct)\cv) .
 \end{multline}
\end{lemma}

\begin{proof}
 Let $\cy$ be $\ct$ or $\cx^i$.
 Then $[\cDelta,\cy] = -2\cnabla\cy$ and $[\cDelta,\cnabla\cy] = 0$.
 Thus
 \begin{align}
  \label{eqn:Deltak-y-commutator} [ \cDelta^k , \cy ] & = -2k(\cnabla\cy)\cDelta^{k-1} , \\
  \notag \cy\cDelta^k\bigl( \cu(\cnabla\cy)\cv \bigr) & = \cDelta^k\bigl( \cu (\cy\cnabla\cy)\cv \bigr) + 2k\cDelta^{k-1}\bigl(\cnabla\cy\bigr)\bigl( \cu(\cnabla\cy)\cv \bigr) .
 \end{align}
 Observe that the infinitesimal generator of dilation is the Euler vector field
 \begin{equation*}
  X = \ct\partial_{\ct} + \sum_{i=0}^n\cx^i\partial_{\cx^i} = -\ct\cnabla\ct + \sum_{i=0}^n \cx^i\cnabla\cx^i .
 \end{equation*}
 Therefore $Q = \lv \cx \rv^2 - \ct\,{}^2$.
 On the one hand,
 \begin{align*}
  \sum_{i=0}^n \cx^i\ct[\cDelta^k,\cx^i](\cu) & = -2k\ct(X + \ct\cnabla\ct)\cDelta^{k-1}\cu , \\
  \sum_{i=0}^n (\cx^i)^2[\cDelta^k,\ct](\cu) & = -2k(Q + \ct^2)(\cnabla\ct)\cDelta^{k-1}\cu .
 \end{align*}
 This establishes Equation~\eqref{eqn:commutator-building-block-one-laplacian}.
 On the other hand,
 \begin{align*}
  \sum_{i=0}^n \cx^i\cDelta^k\bigl(\cu(\cnabla\cx^i)\cv\bigr) & = \sum_{i=0}^n \left( \cDelta^k\bigl(\cu(\cx^i\cnabla\cx^i)\cv\bigr) + 2k\cDelta^{k-1}\bigl(\cnabla\cx^i\bigr)\bigl(\cu(\cnabla\cx^i)\cv\bigr) \right) \\
   & = \cDelta^k(\cu X\cv) - 2k\cDelta^{k-1}\cdelta(\cu\,\cd\cv) + \ct\,\cDelta^k(\cu(\cnabla\ct\,)\cv) , \\
   \sum_{i=0}^n (\cx^i)^2\cDelta^k\bigl(\cu(\cnabla\ct)\cv\bigr) & = (Q + \ct^2)\cDelta^k\bigl(\cu(\cnabla\ct\,)\cv\bigr) .
 \end{align*}
 Combining this with Equation~\eqref{eqn:delta-to-Delta} yields Equation~\eqref{eqn:commutator-building-lbock-two-laplacian}.
\end{proof}

We illustrate how these lemmas recover known commutators.

Our first examples are the GJMS operators of Example~\ref{ex:gjms-operator}.

\begin{example}[GJMS operators (cf.\ \cite{Case2019fl}*{Theorem~1.3})]
 \label{ex:gjms-commutator}
 Let $L_{2k}^{d\theta^2}$ denote the GJMS operator of order $2k<n$ on the round $n$-sphere.
 Let $\cu \in \cmE\bigl[-\frac{n-2k+2}{2}\bigr]$ be a homogeneous extension of a function $u \in C^\infty(S^n)$.
 Lemmas~\ref{lem:ambient-commutator-computation} and~\ref{lem:commutator-building-blocks} yield
 \begin{align*}
  \sum_{i=0}^n x^i [ L_{2k},x^i](u) & = \sum_{i=0}^n \left. \cx^i\Bigl( \ct\bigl[\cDelta^k,\cx^i\bigr](\cu) - \cx^i\bigl[ \cDelta^k,\ct \bigr](\cu) \Bigr) \right|_{(S^n,d\theta^2)} \\
   & = \left. -2k\ct X \cDelta^{k-1}\cu \right|_{(S^n,d\theta^2)} .
 \end{align*}
 Since $\cDelta^{k-1}\cu \in \cmE\bigl[-\frac{n+2k-2}{2}\bigr]$ is a homogeneous extension of $L_{2k-2}u$, we see that
 \begin{equation*}
  \sum_{i=0}^n x^i[L_{2k},x^i] = k(n+2k-2)L_{2k-2} .
 \end{equation*}
 We conclude from Proposition~\ref{prop:factorization} that
 \begin{equation*}
  L_{2k} - \frac{n-2k}{4k}\sum_{i=0}^n x^i[L_{2k},x^i] = L_{2k} - \frac{(n-2k)(n+2k-2)}{4}L_{2k-2} = \Delta L_{2k-2}
 \end{equation*}
 is nonnegative.
 In particular, $L_{2k}$ satisfies the Frank--Lieb Property.
 \qed
\end{example}

Our next examples are the Ovsienko--Redou operators of Example~\ref{ex:ovsienko-redou-operator}.
These are easily treated using the commutator of a building block from Lemma~\ref{lem:case-yan}:

\begin{lemma}
 \label{lem:case-yan-commutator}
 Let $(\cmG,\cg)$ be the flat ambient space of the standard conformal $n$-sphere and let $\cf \in \cmE[w^\prime]$.
 Let $k \in \bN$ and $w \in \bR$.
 Then
 \begin{multline*}
  \sum_{i=0}^n x^i \bigl[D^{d\theta^2}_{2k,w,\cf},x^i\bigr]  = -k(w'-k+1)(n+2k-2)D_{2k-2,w-1,\cf} \\
   + 2k(k-1)\left(\frac{n-2k}{2}+w+w'\right)\left(\frac{n-2k}{2}+w\right)D_{2k-4,w-1,\cDelta\cf} .
 \end{multline*}
\end{lemma}

\begin{proof}
 Lemma~\ref{lem:ambient-commutator-computation}, Lemma~\ref{lem:commutator-building-blocks}, and Equation~\eqref{eqn:Deltak-y-commutator} imply that
 \begin{multline*}
  \sum_{i=0}^n x^i [ D^{d\theta^2}_{2k,w,\cf} , x^i ](u) = -2\sum_{a+b=k} A_{a,b} \Bigl( a(w+w'-2k-b+1)\cDelta^{a-1}(\cf\cDelta^b\cu) \\
   + b(w-k-b+1)\cDelta^a(\cf\cDelta^{b-1}\cu) + ab\cDelta^{a-1}\bigl( (\cDelta\cf)\cDelta^{b-1}\cu \bigr) \Bigr) \Bigr|_{(S^n,d\theta^2)} ,
 \end{multline*}
 where
 \begin{equation*}
  A_{a,b} := \frac{k!}{a!b!} \frac{ \Gamma\bigl( \frac{n-2k}{2} + w + a \bigr) \Gamma\bigl( b - w - w' - \frac{n-2k}{2}\bigr)}{ \Gamma\bigl( \frac{n-2k}{2}+w \bigr) \Gamma\bigl( -w-w'-\frac{n-2k}{2}\bigr)} .
 \end{equation*}
 Reindexing and comparing to Equation~\eqref{eqn:case-yan-operator} yields the conclusion.
\end{proof}

\begin{example}[Ovsienko--Redou operators (cf.\ \cite{CaseLinYuan2022or}*{Theorem~1.5})]
 \label{ex:ovsienko-redou-commutator}
 Let $D_{2k}^{d\theta^2}$ denote the Ovsienko--Redou operator of homogeneity $-2k$ on the round $n$-sphere.
 Combining Equation~\eqref{eqn:or-to-cy} with Lemma~\ref{lem:case-yan-commutator} and simplifying the result yields
 \begin{equation*}
  \sum_{i=0}^n x^i[D^{d\theta^2}_{2k} , x^i]_1 = \frac{k(n+2k-2)(n+k-3)}{3}D_{2k-2}' ,
 \end{equation*}
 where $D_{2k-2}'$ is as in Example~\ref{ex:non-fsa}.
 From here one can directly verify that $(D_{2k},\{1\})$ satisfies the Frank--Lieb Property if $k \in \{ 1, 2 \}$;
 the cases $k \geq 3$ are open.
 \qed
\end{example}

Our third example is the $\sigma_2$-operator of Example~\ref{ex:sigma2-operator}.

\begin{example}[The $\sigma_2$-operator (cf.\ \cite{Case2019fl}*{Theorem~1.4})]
 \label{ex:sigma2-commutator}
 Let $D^{d\theta^2}$ be the $\sigma_2$-operator on the round $n$-sphere.
 Let $\cu \in \cmE\bigl[-\frac{n}{4}\bigr]$ and $\cv \in \cmE\bigl[-\frac{n-4}{4}\bigr]$ be homogeneous extensions of functions $u,v \in C^\infty(S^n)$.
 Using Equations~\eqref{eqn:Delta-on-three} and~\eqref{eqn:sigma2-operator-to-case-yan} and computing as in Example~\ref{ex:ovsienko-redou-commutator} yields
 \begin{equation*}
  \sum_{i=0}^n x^i[D,x^i](u \otimes v^{\otimes2}) = \frac{n-1}{24}uC(v^{\otimes 2}) ,
 \end{equation*}
 where $C$ is the conformally covariant operator defined holographically by
 \begin{align*}
  \cC(\cv^{\otimes2}) := n\cDelta\cv^2 - 8\cv\cDelta\cv
 \end{align*}
 on $\cmE\bigl[-\frac{n-4}{4}\bigr]^{\otimes2}$.
 The formula~\cite{FeffermanGraham2012}*{Equations~(3.6) and~(4.4)} for the straight and normal ambient metric implies that
 \begin{equation*}
  C(u \otimes v^{\otimes 2}) := n\Delta v^2 - 8v\Delta v + \frac{(n-4)^2}{2}J .
 \end{equation*}
 Direct computation implies that the constrained polydifferential operator $(D,\{C\})$ satisfies the Frank--Lieb Property~\cite{Case2019fl}*{Proof of Theorem~1.2}.
 \qed
\end{example}

We conclude with a brief comment on the $v_k$-operators of Example~\ref{ex:vk}.

\begin{example}[The $v_k$-operators]
 \label{ex:vk-commutator}
 It is not presently known if the $v_k$-operators satisfy the Frank--Lieb Property for some choice of constraint set.
 However, when restricted to locally conformally flat manifolds, Viaclovsky~\cite{Viaclovsky2000} classified critical points in the elliptic $k$-cones via an Obata-type argument.
\end{example}

%% file: examples.tex
\section{Geometric Aubin sets}
\label{sec:examples}

Examples of geometric Aubin sets are known for the GJMS operators and for the operators associated to the renormalized volume coefficients.
In the former case, one requires additional restrictions as the order of the operator grows.
In the latter case, one also requires the renormalized volume coefficient $v_k$ to coincide with the $\sigma_k$-curvature;
hence $k \leq 2$ or the manifold is locally conformally flat~\cite{BransonGover2008}.
One unifying feature of these sets is that they all involve positivity assumptions that can be understood in terms of the first nonlinear eigenvalue of $(D,\sC)$.

\begin{definition}
 \label{defn:spectrum}
 Let $(D,\sC)$ be a constrained polydifferential operator of rank $r$ and homogeneity $-2k>-n$ on $n$-manifolds.
 The \defn{first nonlinear eigenvalue} of $(D,\sC)$ on a compact Riemannian $n$-manifold $(M,g)$ is
 \begin{equation*}
  \lambda_{(D,\sC)}(M,g) := \inf \left\{ \kD^g\bigl(u^{\otimes r}\bigr) \suchthat u \in U_{\sC}^g , \lV u \rV_r = 1 \right\} .
 \end{equation*}
\end{definition}

We denote by $\lambda_D$ the first nonlinear eigenvalue of an unconstrained polydifferential operator $D$.

If $D$ has rank $2$, then $\lambda_D$ recovers the usual definition of the first eigenvalue.
If $(D,\sC)$ is as in Example~\ref{ex:sigma2-commutator}, then $\lambda_{(D,\sC)}$ recovers the nonlinear eigenvalue associated to the $\sigma_2$-curvature as studied by Ge and Wang~\cite{GeWang2006}*{Proposition~3} and by Ge, Lin, and Wang~\cite{GeLinWang2010}*{Lemma~3}.

The signs of the first nonlinear eigenvalue for $(D,\sC)$ and the $(D,\sC)$-Yamabe constant coincide under the assumption of a Sobolev-type inequality for $D$:

\begin{lemma}
 \label{lem:positivity-to-spectrum}
 Let $(D,\sC)$ be a constrained polydifferential operator of rank $r$ and homogeneity $-2k > -n$ on $n$-manifolds.
 Let $(M,g)$ be a compact Riemannian $n$-manifold for which there are constants $A,B >0$ such that
 \begin{equation}
  \label{eqn:sobolev-assumption}
  A\lV u \rV_{r^\ast}^r \leq \kD^g(u^{\otimes r}) + B\lV u \rV_r^r
 \end{equation}
 for all $u \in U_{\sC}^g$.
 Then the signs of $\lambda_{(D,\sC)}(M,g)$ and $Y_{(D,\sC)}(M,[g])$ coincide.
\end{lemma}

\begin{proof}
 Clearly $\lambda_{(D,\sC)}(M,g)<0$ if and only if $Y_{(D,\sC)}(M,[g])<0$.
 
 Suppose that $Y_{(D,\sC)}(M,[g])>0$.
 Let $u \in U_{\sC}^g$ be such that $\lV u \rV_r=1$.
 H\"older's inequality implies that
 \begin{equation*}
  \kD^g(u^{\otimes r}) \geq Y_{(D,\sC)}(M,[g])\lV u \rV_{r^\ast}^r \geq Y_{(D,\sC)}(M,[g])\Vol(M)^{-2k/n} .
 \end{equation*}
 Therefore $\lambda_{(D,\sC)}(M,g) > 0$.
 
 Suppose instead that $\lambda_{(D,\sC)}(M,g) > 0$.
 By homogeneity,
 \begin{equation*}
  \kD^g(u^{\otimes r}) \geq \lambda_{(D,\sC)}(M,g) \lV u \rV_r^r
 \end{equation*}
 for all $u \in U_{\sC}^g$.
 Let $u \in U_{\sC}^g$ be such that $\lV u \rV_{r^\ast}=1$.
 Condition~\eqref{eqn:sobolev-assumption} implies that
 \begin{equation*}
  A \leq \left( 1 + \frac{B}{\lambda_{(D,\sC)}(M,g)} \right)\kD^g(u^{\otimes r}) .
 \end{equation*}
 Therefore $Y_{(D,\sC)}(M,[g])>0$.
\end{proof}

Note that Condition~\eqref{eqn:sobolev-assumption} was only used to prove that if $\lambda_{(D,\sC)}(M,g) > 0$, then $Y_{(D,\sC)}(M,[g])>0$.
As in Druet and Hebey's $AB$-Program~\cite{DruetHebey2002}, it is interesting to determine the best constants $A$ and $B$ in Condition~\eqref{eqn:sobolev-assumption}, when they exist.

We begin our discussion of examples with the conformal Laplacian, whose best geometric Aubin set comes from the resolution of the Yamabe Problem~\cite{LeeParker1987}:

\begin{example}[The conformal Laplacian]
 \label{ex:yamabe-set}
 Fix $n \geq 3$ and let $L_2$ denote the conformal Laplacian.
 The $L_2$-Yamabe constant is the usual Yamabe constant.
 Since a compact conformal $n$-manifold has positive Yamabe constant if and only if it admits a metric with positive scalar curvature~\cite{Schoen1989}*{Lemma~1.2}, the set
 \begin{equation*}
  \sA_{n,2} := \left\{ (M^n,\kc) \suchthat Y_{L_2}(M,\kc) > 0 \right\}
 \end{equation*}
 is closed under finite coverings.
 Results of Aubin~\cite{Aubin1976} and Schoen~\cite{Schoen1984} imply that $\sA_{n,2}$ is a geometric Aubin set for $L_2$.
 Applying Theorem~\ref{thm:compact-nonuniqueness} recovers nonuniqueness results of Schoen~\cite{Schoen1989} and of Bettiol and Piccione~\cite{BettiolPiccione2018}.
 \qed
\end{example}

In dimensions $n>4$, early partial solutions of the Yamabe problem for the Paneitz operator assume the positivity of the Yamabe constant and the existence of a metric with positive fourth-order $Q$-curvature~\cites{GurskyMalchiodi2014,HangYang2016}.
More recently, Gursky, Hang, and Lin~\cite{GurskyHangLin2016} showed that the positivity of $L_2$ and $L_4$ are sufficient to solve the $L_4$-Yamabe Problem in dimensions $n \geq 6$:

\begin{example}[The Paneitz operator]
 \label{ex:paneitz-set}
 Fix $n \geq 6$ and let $L_4$ denote the Paneitz operator.
 The $L_4$-Yamabe constant coincides with the Yamabe-type constant denoted $Y_4(M^n,g)$ by Gursky, Hang, and Lin~\cite{GurskyHangLin2016}*{Equation~(1.12)}.
 Their work, as shown by Andrade et al.~\cite{AndradeCasePiccioneWei2023}*{Proposition~3.2}, implies that
 \begin{equation*}
  \sA_{n,4} := \left\{ (M^n,\kc) \suchthat Y_{L_2}(M,\kc) , Y_{L_4}(M,\kc) > 0 \right\}
 \end{equation*}
 is a geometric Aubin set for $L_4$.
 Applying Theorem~\ref{thm:compact-nonuniqueness} to $\sA_{n,4}$ recovers nonuniqueness results of Bettiol, Piccione, and Sire~\cite{BettiolPiccioneSire2021}.
 \qed
\end{example}

It is expected that $\sA_{5,4}$ is also a geometric Aubin set~\cite{GurskyHangLin2016}*{p.\ 1532}.
The following conjecture asks whether the obvious generalization of Examples~\ref{ex:yamabe-set} and~\ref{ex:paneitz-set} is true (cf.\ \citelist{ \cite{AndradePiccioneWei2023}*{Conjecture~1} \cite{CaseGover2026}*{Conjecture~4.4}}).

\begin{conjecture}
 \label{conj:gjms-aubin-set}
 Fix $k,n \in \bN$ such that $n > 2k$.
 Then
 \begin{equation}
  \label{eqn:gjms-aubin-set}
  \sA_{n,2k} := \left\{ (M^n,\kc) \suchthat Y_{L_2}(M,\kc), \dotsc, Y_{L_{2k}}(M,\kc) > 0 \right\}
 \end{equation}
 a geometric Aubin set for $L_{2k}$.
\end{conjecture}

By a result of Mazumdar~\cite{Mazumdar2016}*{Theorem~3}, to verify Conjecture~\ref{conj:gjms-aubin-set}, one need only show that if $(M,\kc) \in \sA_{n,2k}$, then $L_{2k}$ satisfies the Strong Maximum Principle---i.e.\ if $L_{2k}u \geq 0$, then $u=0$ or $u>0$---and that $Y_{L_{2k}}(M,\kc) < Y_{L_{2k}}(S^n,\kc_0)$ unless $(M,\kc) = (S^n,\kc_0)$.
Beckner~\cite{Beckner1993}*{Theorem~6} identified the extremals of the constant $Y_{L_{2k}}(S^n,\kc_0)$;
they are also identified by the Frank--Lieb argument.
At present, the best partial solutions to Conjecture~\ref{conj:gjms-aubin-set} currently known require additional structural assumptions on the conformal manifold.
For example:

\begin{example}[The higher-order GJMS operators]
 \label{ex:gjms-set}
 Let $L_{2k}$ denote the GJMS operator of order $2k$.
 Results of Case and Malchiodi~\cite{CaseMalchiodi2023} imply that for each $m \in \bN \cup \{ 0 \}$ there is an $N = N(k,m)$ such that if $n \geq N(k,m)$, then
 \begin{equation*}
  \sA := \left\{ (M^n, \kc) \suchthat \text{$(M^n,\kc)$ is virtually conformal to an $m$-SEP} \right\}
 \end{equation*}
 is a geometric Aubin set for $L_{2k}$~\cite{AndradeCasePiccioneWei2023}*{Proposition~3.4}.
 Applying Theorem~\ref{thm:compact-nonuniqueness} in the case $k=3$ recovers a nonuniqueness result of Andrade, Piccione, and Wei~\cite{AndradePiccioneWei2023}.
 \qed
\end{example}

Here an \defn{$m$-SEP}, or $m$-special Einstein product~\cite{GoverLeitner2009}, is the Riemannian product of Einstein manifolds $(M_1^m,g_1)$ and $(M_2^{n-m},g_2)$ with $\Ric_{g_1} = -(m-1)g_1$ and $\Ric_{g_2} = (n-m-1)g_2$.
We say that $(M^n,\kc)$ is \defn{virtually} conformal to an $m$-SEP if there is a finite connected covering $\pi \colon \cM \to M$ for which there is a $\cg \in \pi^\ast\kc$ such that $(\cM,\cg)$ is isometric to an $m$-SEP.

A specialization of Conjecture~\ref{conj:gjms-aubin-set} is to the optimal dimension in Example~\ref{ex:gjms-set}:

\begin{conjecture}
 \label{conj:case-malchiodi}
 $N(k,m) = 2k + 2m - 1$.
\end{conjecture}

Case and Malchiodi~\cite{CaseMalchiodi2023}*{Lemma~4.3} verified Conjecture~\ref{conj:case-malchiodi} for locally conformally flat manifolds.
Otherwise, when $m>0$ the value of $N(k,m)$ produced by their proof is much larger than $2k+2m-1$.
If Conjecture~\ref{conj:gjms-aubin-set} is true, then a computation of Case and Malchiodi~\cite{CaseMalchiodi2023}*{Theorem~1.1} verifies Conjecture~\ref{conj:case-malchiodi}.

Work on the $\sigma_k$-curvature produces geometric Aubin sets for the operators associated to the renormalized volume coefficients~\cites{GuanWang2004,GeWang2006}.
We detail the case of the $\sigma_2$-curvature:

\begin{example}[The $\sigma_2$-operator]
 \label{ex:sigma2-set}
 Fix $n \geq 5$, let $D$ denote the $\sigma_2$-operator, and let $C$ be as in Example~\ref{ex:sigma2-commutator}.
 Set $\sC_1 := \{ C \}$ and $\sC_2 := \{ C, D \}$.
 The positive elliptic $1$- and $2$-cones are $U_{\sC_1}$ and $U_{\sC_2}$, respectively.
 Ge, Lin, and Wang~\cite{GeLinWang2010}*{Theorem~1} showed that if $(M,\kc)$ is a compact conformal $n$-manifold with positive Yamabe constant and $\lambda_{(D,\sC_1)}>0$, then $U_{\sC_2}^g \not= \emptyset$ for each $g \in \kc$.
 Combining this with an existence result of Guan and Wang~\cite{GuanWang2004}*{Theorem~1(A)} implies that the set
 \begin{equation*}
  \sA' := \left\{ (M,\kc) \suchthat W = 0 , Y_{L_2}(M,\kc) > 0 , Y_{(D,\sC_1)}(M,\kc) > 0 \right\}
 \end{equation*}
 of compact, locally conformally flat $n$-manifolds with positive $L_2$- and $(D,\sC_1)$-Yamabe constants is a geometric Aubin set for $(D,\sC_2)$.
 Combining this instead with a result of Ge and Wang~\cite{GeWang2006}*{Proposition~1} when $n \geq 9$ implies that the set
 \begin{equation*}
  \sA'' := \left\{ (M,\kc) \suchthat Y_{L_2}(M,\kc), Y_{(D,\sC_1)}(M,\kc) > 0 \right\}
 \end{equation*}
 of compact conformal $n$-manifolds with positive $L_2$- and $(D,\sC_1)$-Yamabe constants is a geometric Aubin set for $(D,\sC_2)$.
 Applying Theorem~\ref{thm:compact-nonuniqueness} to these sets establishes the nonuniqueness of metrics with constant $\sigma_2$-curvature in all dimension $n \geq 5$, generalizing an observation of Viaclovsky~\cite{Viaclovsky2000}*{Section~5} for $S^1 \times S^{n-1}$.
 \qed
\end{example}

Ge and Wang~\cite{GeWang2025} announced that $\sA''$ is a geometric Aubin set for $(D,\sC_1)$ in all dimensions $n \geq 5$.

We conclude by commenting on the $v_k$-operators:

\begin{example}[The $v_k$-operators]
 \label{ex:vk-set}
 Locally conformally flat manifolds with nonempty positive $k$-cone form a geometric Aubin set for the $v_k$-operators~\cite{GuanWang2004}.
\end{example}

%% file: nonuniqueness.tex
\section{A stronger nonuniqueness result}
\label{sec:nonuniqueness}

Theorems~\ref{thm:compact-nonuniqueness} and~\ref{thm:noncompact-nonuniqueness} are true for a weaker notion of geometric Aubin sets that assumes only the existence of critical points of the corresponding constrained functional with uniformly bounded energy.
We present the generalization of Theorem~\ref{thm:compact-nonuniqueness} that, for concreteness, constructs geometrically distinct solutions of
\begin{equation}
 \label{eqn:normalized-pde}
 D^g(u^{\otimes\ell}) = \lv u \rv^{\frac{(r-2)n+4k}{n-2k}}u .
\end{equation}
Combining this with the covering argument of Andrade et al.~\cite{AndradeCasePiccioneWei2023}*{Lemma~4.2} gives the generalization of Theorem~\ref{thm:noncompact-nonuniqueness}.

\begin{theorem}
 \label{weaker-nonuniqueness}
 Let $(D,\sC)$ be a constrained polydifferential operator of rank $r$ and homogeneity $-2k>-n$ on $n$-manifolds.
 Suppose that $\sA$ is a set of compact conformal $n$-manifolds such that
 \begin{enumerate}
  \item if $(M,\kc) \in \sA$, then $Y_{(D,\sC)}(M,\kc) > 0$;
  \item there is a constant $\Lambda>0$ such that if $(M,\kc) \in \sA$ and $g \in \kc$, then there is a solution $u \in U_{\sC}^g$ of PDE~\eqref{eqn:normalized-pde} for which $\lV u \rV_{r^\ast} \leq \Lambda$; and
  \item if $(M,\kc) \in \sA$ and $\pi \colon \cM \to M$ is a finite covering, then $(\cM,\pi^\ast\kc) \in \sA$.
 \end{enumerate}
 Let $(M,\kc) \in \sA$.
 If $\pi_1(M)$ has infinite profinite completion, then for each $m \in \bN$ there is a finite regular covering $\pi \colon \cM \to M$ such that for any $g \in \kc$ there are at least $m$ pairwise geometrically distinct solutions of PDE~\eqref{eqn:pde} in $U_{\sC}^{\pi^\ast g}$.
\end{theorem}

\begin{proof}
 We claim that there is an infinite tower of finite regular coverings
 \begin{equation*}
  \dotsm \overset{\pi_4}{\longrightarrow} \cM_3 \overset{\pi_3}{\longrightarrow} \cM_2 \overset{\pi_2}{\longrightarrow} \cM_1 := M
 \end{equation*}
 and a sequence of solutions $u_j \in U_{\sC}^{\Pi_j^\ast g}$ of PDE~\eqref{eqn:normalized-pde} with $\lV u_j \rV_{r^\ast} \leq \Lambda$ such that
 \begin{equation}
  \label{eqn:key-estimate}
  \mF^{\Pi_j^\ast g}(u_j) < \mF^{\Pi_{j}^\ast g}(\pi_j^\ast u_{j-1})
 \end{equation}
 for each $j \in \bN$, where
 \begin{equation*}
  \Pi_j := \pi_2 \circ \dotsm \circ \pi_j \colon \cM_j \to \cM_1 ,
 \end{equation*}
 with the convention $\Pi_1 := \Id$.
 If true, then Inequality~\eqref{eqn:key-estimate} implies that
 \begin{equation*}
  \mF^{\Pi_m^\ast g}(u_m) < \mF^{\Pi_m^\ast g_m}(\pi_m^\ast u_{m-1}) < \mF^{\Pi_m^\ast g_m}(\pi_m^\ast \pi_{m-1}^\ast u_{m-2}) < \dotsm < \mF^{\Pi_m^\ast g}(\Pi_m^\ast u_1) .
 \end{equation*}
 Therefore $\{ u_m , \pi_m^\ast u_{m-1} , \pi_m^\ast \pi_{m-1}^\ast u_{m-2} , \dotsc , \Pi_m^\ast u_1 \}$ are geometrically distinct solutions of PDE~\eqref{eqn:normalized-pde} on $(\cM_m,\Pi_m^\ast g)$.
 
 We now prove the claim.
 To that end, observe that if $(\hM,\widehat{\kc}) \in \sA$ and $\widehat{u} \in U_{\sC}^{\hg}$, $\hg \in \widehat{\kc}$, is a solution of PDE~\eqref{eqn:normalized-pde} with $\lV \widehat{u} \rV_{r^\ast} \leq \Lambda$, then
 \begin{equation}
  \label{eqn:energy-estimate}
  \mF^{\hg}(\widehat{u}) = \lV \widehat{u} \rV_{r^\ast}^{\frac{2kr}{n-2k}} \leq \Lambda^{\frac{2kr}{n-2k}} .
 \end{equation}

 Since $(\cM_1,\Pi_1^\ast\kc) \in \sA$, there is a solution $u_1 \in U_{\sC}^{\Pi_1^\ast g}$ of PDE~\eqref{eqn:normalized-pde} such that $\lV u_1 \rV_{r^\ast} \leq \Lambda$.
 Suppose that a finite connected covering $\Pi_j \colon \cM_j \to \cM_1$ and a solution $u_j \in U_{\sC}^{\Pi_j^\ast g}$ of PDE~\eqref{eqn:normalized-pde} with $\lV u_j \rV_{r^\ast} \leq \Lambda$ are given.
 Conditions~(i) and~(iii) yield
 \begin{equation*}
  \mF^{\Pi_j^\ast g}(u_j) \geq Y_{(D,\sC)}(\cM_j,\Pi_j^\ast\kc) > 0 .
 \end{equation*}
 Since $\pi_1(M)$ has infinite profinite completion, so too does $\cM_j$.
 Pick a finite connected covering $\pi_{j+1} \colon \cM_{j+1} \to \cM_j$ of degree
 \begin{equation*}
  d_{j+1} > \Lambda^{\frac{rn}{n-2k}} \bigl( Y_{(D,\sC)}(\cM_j,\Pi_j^\ast\kc) \bigr)^{-\frac{n}{2k}} .
 \end{equation*}
 Set $\Pi_{j+1} := \Pi_j \circ \pi_{j+1}$.
 On the one hand, direct computation yields
 \begin{equation}
  \label{eqn:covered-energy-estimate}
  \mF^{\Pi_{j+1}^\ast g}(\pi_{j+1}^\ast u_j) = d_{j+1}^{\frac{2k}{n}}\mF^{\Pi_j^\ast g}(u_j) > \Lambda^{\frac{2kr}{n-2k}} .
 \end{equation}
 On the other hand, Condition~(ii) implies that there is a solution $u_{j+1} \in U_{\sC}^{\Pi_{j+1}^\ast g}$ of PDE~\eqref{eqn:normalized-pde} with $\lV u_{j+1} \rV_{r^\ast} \leq \Lambda$.
 Combining Inequalities~\eqref{eqn:energy-estimate} and~\eqref{eqn:covered-energy-estimate} yields
 \begin{equation*}
  \mF^{\Pi_{j+1}^\ast g}(u_{j+1}) \leq \Lambda^{\frac{2kr}{n-2k}} < \mF^{\Pi_{j+1}^\ast g}(\pi_{j+1}^\ast u_j) . \qedhere
 \end{equation*}
\end{proof}

%% file: future.tex
\section{Future directions}
\label{sec:future}

The applicability of Theorems~\ref{thm:compact-nonuniqueness} and~\ref{thm:noncompact-nonuniqueness} depends on being able to identify geometric Aubin sets.
Doing so is arguably the most important task in the general study of formally self-adjoint, conformally covariant, polydifferential operators\footnote{
 \label{cr-footnote}
 It would also be interesting to develop the parallel theory in CR geometry, building on the analogy between the Yamabe Problem and the CR Yamabe Problem~\cite{JerisonLee1987}.
 A theory that includes ``primed operators'', such as the $P^\prime$-operator~\cites{CaseYang2012,Hirachi2013} and multiplication by the $I^\prime$-curvatures~\cites{CaseGover2013,CaseTakeuchi2019,Marugame2019,Marugame2020}, is also of interest.}.
We conclude this paper with a series of objectives that aim to complete this task.

The $(D,\sC)$-Yamabe Problem is an analytic problem on manifolds, and as such requires detailed knowledge of the operator $D$.
While Theorem~\ref{thm:cly} constructs formally self-adjoint, conformally covariant, polydifferential operators, it says little about their analytic properties.
Our first objective aims to address this issue by classifying Weyl operators.

\begin{objective}
 \label{o:classification}
 Show that every conformally covariant polydifferential operator of homogeneity $-2k > -n$ on $n$-manifolds is a Weyl operator.
\end{objective}

There are two partial results towards Objective~\ref{o:classification}.
First, a result of Bailey, Eastwood, and Graham~\cite{BaileyEastwoodGraham1994} achieves Objective~\ref{o:classification} for operators of rank $1$.
Second, Alexakis~\cite{Alexakis2003} showed, by adapting the arguments of Bailey et al., that any conformally covariant polydifferential operator of homogeneity $-2k > -n$ on $n$-manifolds can be realized by applying an ambient operator to a \emph{formally harmonic} extension of its input.
One possible approach to Objective~\ref{o:classification} is to take an operator $D$ produced by Alexakis' result, compute its commutators with $Q$, and compare with Equation~\eqref{eqn:sl2-algebra} to find a tangential modification $D'$ of $D$ by adding a suitable operator with $\cDelta$ as a right factor.

Solving the $(D,\sC)$-Yamabe Problem requires identifying constraints.
Example~\ref{ex:sigma2-commutator} indicates the role of the Frank--Lieb Property in this task.

\begin{objective}
 \label{o:commutator}
 Classify Weyl operators for which there is a finite set $\sC$ of polydifferential operators such that $(D,\sC)$ satisfies the Frank--Lieb Property.
\end{objective}

The focus on Weyl operators is motivated by the observations of Section~\ref{sec:ambient} and Section~\ref{sec:frank-lieb} that formulas for $D$ and $\sum x^i[D,x^i]$, respectively, are straightforward to compute holographically.
It is not presently clear how best to verify Condition~\eqref{eqn:fl-poincare}.
One expects that the elements of $\sC$ are Weyl operators with rank at most that of $D$ and homogeneity strictly greater than that of $D$.

As discussed in Remark~\ref{rk:fl-regularity}, the Frank--Lieb Property sometimes establishes the regularity of solutions of PDE~\eqref{eqn:pde}.
Our next objective seeks the right space in which to construct minimizers.

\begin{objective}
 \label{o:existence}
 Determine which constrained polydifferential operators $(D,\sC)$ as in Objective~\ref{o:commutator} have the property that there is a completion $\sB^{d\theta^2}$ of $U_{\sC}^{d\theta^2}$ for which there is a $u$ in the interior $\Int\sB^{d\theta^2}$ of $\sB^{d\theta^2}$ satisfying
 \begin{align*}
  \kD^{d\theta^2}(u^{\otimes r}) & = Y_{(D,\sC)}(S^n,\kc_0) , \\
  \lV u \rV_{r^\ast} & = 1 .
 \end{align*}
\end{objective}

Ideally, the topology on $\sB$ will be sufficiently strong that one can deduce, as in Remark~\ref{rk:fl-regularity}, that $u$ is in fact smooth.
For the unconstrained GJMS operator $L_{2k}$, Objective~\ref{o:existence} is achieved with $\sB = W^{k,2}$.

One should look to achieve Objective~\ref{o:existence} using a ``natural'' completion.
A first check that the right choice has been made is to adapt Aubin's existence result~\cite{Aubin1976} to constrained polydifferential operators.

\begin{objective}
 \label{o:general-existence}
 Determine which constrained polydifferential operators $(D,\sC)$ as in Objective~\ref{o:commutator} have the property that there is an assignment $\sB$ to each compact Riemannian $n$-manifold $(M,g)$ of a completion $\sB^g$ of $U_{\sC}^{g}$ such that
 \begin{enumerate}
  \item if $\hg = e^{2\Upsilon} g$, then $\sB_{\sC}^{\hg} = e^{\frac{n-2k}{r}\Upsilon}\sB_{\sC}^g$;
  \item if $\pi \colon \hM \to M$ is a finite covering, then $\pi^\ast(\sB^g) \subseteq \sB^{\pi^\ast g}$; and
  \item if $Y_{(D,\sC)}(M,[g]) < Y_{(D,\sC)}(S^n,\kc_0)$, then there is a $u \in \Int \sB^g$ satisfying
  \begin{align*}
   \kD^{g}(u^{\otimes r}) & = Y_{(D,\sC)}(M,[g]) , \\
   \lV u \rV_{r^\ast} & = 1 .
  \end{align*}
 \end{enumerate}
\end{objective}

This represents the first step towards finding a geometric Aubin set for $(D,\sC)$.
The next step requires proving regularity, and is probably best considered first in two special cases.
First are the GJMS operators, where the missing ingredient is the Strong Maximum Principle;
cf.\ Conjecture~\ref{conj:gjms-aubin-set}.
Second are the $v_k$-operators, as the corresponding operator is second order~\cite{Graham2009}*{Theorem~1.1} and the requisite regularity theory has been worked out in the locally conformally flat case~\cites{GuanWang2003,ShengTrudingerWang2007}.